\documentclass{svjour3}
\usepackage{amssymb}

\usepackage{epic}
\usepackage{amsmath}
\usepackage{amsfonts}
\usepackage{pstricks}

\usepackage{float}
\usepackage{subfig}
\usepackage{graphicx}
\usepackage{multirow}
\usepackage{adjustbox}
\usepackage{color}
\usepackage[subfigure]{tocloft}
\usepackage{subfig}
\usepackage[utf8]{inputenc}
\usepackage{soul}
\usepackage{hyperref}
\hypersetup{colorlinks=true,linkcolor=blue,urlcolor=blue}

\def\bigO{\mathcal{O}}
\def\mig{\frac12}
\def\qRl{\mu}

\def\mn{  \left(
    \begin{array}{c}
      s\\
      m
    \end{array}      
    \right)
}
\def\mna#1#2{  \left(
    \begin{array}{c}
      #1\\
      #2
    \end{array}      
    \right)
}

\def\bigO{\mathcal{O}}
\journalname{Journal of Scientific Computing}

\begin{document}
\title{An approximate Lax-Wendroff-type procedure for high order
  accurate schemes for hyperbolic conservation laws} 
\titlerunning{An approximate Lax-Wendroff-type procedure}
\author{D.~Zorío \and A.~Baeza \and P.~Mulet}
\date{}

\institute{
Departament de Matemàtiques,  Universitat de
  València    (Spain); emails: david.zorio@uv.es, antonio.baeza@uv.es, mulet@uv.es. 
}

\leavevmode\thispagestyle{empty}

\noindent This version of the article has been accepted for publication, after a peer-review process, and is subject to Springer Nature’s AM terms of use, but is not the Version of Record and does not reflect post-acceptance improvements, or any corrections. The Version of Record is available online at: \url{https://doi.org/10.1007/s10915-016-0298-2}

\newpage

\maketitle

\begin{abstract}
  A  high order time stepping applied to spatial discretizations
  provided by the method of lines for hyperbolic conservations
  laws is presented. This procedure is related to the
  one  proposed in
  [Qiu, Shu, SIAM J. Sci. Comput., 2003] for numerically solving hyperbolic conservation laws. 
  Both methods are based on the conversion of time derivatives 
  to spatial derivatives through a Lax-Wendroff-type procedure, also
  known as 
  Cauchy-Kovalevskaya process. The original approach in  [Qiu, Shu, SIAM J. Sci. Comput., 2003]
  uses the exact expressions of the fluxes and their derivatives whereas 
  the new procedure computes suitable finite difference approximations
  of them  ensuring arbitrarily high order accuracy both in space and
  time as the
  original technique does, with a much simpler implementation and
  generically better   performance, since only flux evaluations are
  required and no symbolic computations of flux derivatives are needed.

\end{abstract}

\section{Introduction}
The design of stable high order accurate schemes for hyperbolic
conservation laws following the method of lines approach is a
challenging problem. After the
successful ENO  
\cite{Harten1987,ShuOsher1988,ShuOsher1989} and WENO \cite{JiangShu96}
spatial semidiscretization
techniques, which
achieve arbitrarily high order spatial accuracy with excellent
results in terms of accuracy and performance, many research lines have
been focused on developing  high order accurate time
discretizations as well. There are other approaches, for instance ADER
schemes (see \cite{Toro2009} and references therein).

A commonly used strategy to obtain high order time discretizations 
 is the implementation of SSP Runge-Kutta schemes (as the popular TVD RK3), as
they provide high resolution approximations  and are stable 
with respect to the Total Variation under the same CFL conditions as
the forward Euler method. The TVD RK3 ODE solver achieves third order
accuracy with low storage requirements and a large CFL condition (see
\cite{ShuOsher1988}). For orders 
higher than three the  
stability properties are not so favorable, forcing either 
a smaller CFL restriction or the inclusion of additional stages in the
algorithm leading to an efficiency loss in both cases.

In order to simplify the formulation of high order accurate schemes in
time, Qiu and Shu \cite{QiuShu2003} developed a new time
discretization following a
Lax-Wendroff-type procedure and based in the Cauchy-Kovalevskaya idea, where
the numerical solution at a further time step is computed by a Taylor expansion
in time with the time derivatives transformed into spatial derivatives
using the equations.
The main benefit of such scheme with
respect to the Runge-Kutta schemes is that
only one WENO reconstruction procedure with flux-splitting is required
to be performed at each spatial cell per time step, regardless of the
order of the method, yielding an overall better performance.

However,
the main drawback of this approach is that the number of terms to
compute increases exponentially as the order increases, resulting
in very complex expressions that require the use of symbolic computation 
packages and involve a high computational cost.

In this paper we focus in developing an alternative version which,
instead of computing the exact expressions of the time derivatives of
the fluxes, approximates them through high order central
divided difference formulas. The main goal in the
formulation of this alternative version is to develop a scheme where
no exact flux derivatives are required, yielding a
straightforward implementation and a better performance in cases
where the complexity of the flux derivatives affects severely the
performance of the computation of high order terms in the exact
procedure.

This paper is organized as follows: In Section \ref{ens} we show the
details about the general framework of the equations we work with and
a general overview of the exact Lax-Wendroff-type procedure, Section
\ref{cka} stands for the formulation of the approximate Lax-Wendroff-type
procedure, where some important properties, such as the achievement of
the desired accuracy order and the conservation form, are
proven. Several numerical experiments comparing both techniques are
presented in Section \ref{nex} and some conclusions are drawn in
Section \ref{cnc}.

\section{Original Lax-Wendroff-type procedure}\label{ens}

The PDEs considered in this work are systems of
$d$-dimensional $m$ hyperbolic conservation laws:
\begin{equation}\label{hcl}
  \begin{aligned}
    u_t+\nabla \cdot f(u)&=0,
\end{aligned}
\end{equation}
where $ \nabla \cdot $ denotes the divergence operator with respect to
the spatial variables $x_1,\dots,x_d$ and 
$$u=u(x;t)\in\mathbb{R}^m,\quad
x=(x_1,\ldots,x_d),\quad
(x;t)\in\Omega\times\mathbb{R}^+\subseteq\mathbb{R}^d\times\mathbb{R}^+,
\quad f^{i}:\mathbb{R}^m\rightarrow\mathbb{R}^m,$$
with
$$u=\left[\begin{array}{c}
    u_1 \\
    \vdots \\
    u_m
    \end{array}\right],\quad f^i=\left[\begin{array}{c}
      f^i_1 \\
      \vdots \\
      f^i_m
      \end{array}\right],
    \quad f=\begin{bmatrix}f^1&\dots&f^d
      \end{bmatrix}.
      $$
System \eqref{hcl} is complemented with initial conditions $u(x;0)=u_0(x)$, $x\in\Omega$, and prescribed boundary
conditions. 

For the sake of simplicity, we start with the one-dimensional scalar
case ($d=m=1$). For the solution $u(x, t)$ of $u_t+f(u)_x=0$ on a fixed spatial
   grid $\{x_i\}$ with constant spacing $h=x_{i+1}-x_{i}$ and some time $t_n$
   from a temporal grid with spacing $\delta=\Delta t=t_{n+1}-t_n>0$,
   proportional to $h$, $\delta=\tau h$, where $\tau$ is dictated by
   stability restrictions (CFL condition), we use
the following notation for time derivatives of $u$ and $f(u)$:
   \begin{align*}
     u_{i,n}^{(l)}&=\frac{\partial^{l} u(x_i, t_n)}{\partial t^l}\\
     f_{i,n}^{(l)}&=\frac{\partial^{l} f(u)(x_i, t_n)}{\partial t^l}.
   \end{align*}

Our goal is to obtain an $R$-th order accurate numerical scheme, i.e., a scheme
with a local truncation error 
of order $R+1$, based on the Taylor expansion  of the solution $u$
from time $t_n$ to the next time $t_{n+1}$:
$$u_i^{n+1}=\sum_{l=0}^R\frac{\Delta
  t^l}{l!}u_{i,n}^{(l)}+\bigO(\Delta t^{R+1}).$$
For this purpose, we aim to define  corresponding approximations
\begin{equation}\label{eq:approximations}
\begin{aligned}
  \widetilde u_{i,n}^{(l)}&=u_{i,n}^{(l)}+\bigO(h^{R+1-l})\\
  \widetilde f_{i,n}^{(l)}&=f_{i,n}^{(l)}+\bigO(h^{R-l}),
\end{aligned}
\end{equation}
by recursion on $l$, assuming  (for a local truncation error
analysis) that  $\widetilde
u^{0}_{i,n}=u^{(0)}_{i,n}=u(x_i, t_n)$.

 The fact that $u$ solves the system of conservation laws implies that 
the time derivatives
$u_{i,n}^{(l)}$, $1\leq l\leq R$, can be written in terms of the spatial
divergence of $f_{i,n}^{(l-1)}$:
\begin{equation}\label{eq:ck}
\frac{\partial^lu}{\partial t^l}=\frac{\partial^{l-1}}{\partial
  t^{l-1}}
\big(u_t\big)=-\frac{\partial^{l-1}}{\partial t^{l-1}}\big(f(u)_x\big)=
-\left[\frac{\partial^{l-1}f(u)}{\partial t^{l-1}}\right]_x,
\end{equation}
and following the Cauchy-Kovalevskaya (or Lax-Wendroff-type for
second order) procedure and using Faà di Bruno's formula 
stated in Theorem \ref{th:faadibruno}, the time derivatives of 
the flux, $f_{i,n}^{(l-1)}$, can in turn be written in terms of 
some functions of $u_{i,n}^{(j)}$, $j<l$,
\begin{equation}\label{eq:35}
f_{i,n}^{(l-1)}=F_{l-1}(u_i^n,u_{i,n}^{(1)},\ldots,u_{i,n}^{(l-1)}).
\end{equation}
Putting together \eqref{eq:ck} and \eqref{eq:35} we conclude that 
the time derivatives $u_{i,n}^{(l)}$ of $u$
 can also be written in terms of 
some functions of its lower order derivatives $u_{i,n}^{(j)}$, $j<l$.

More specifically, assume we have numerical data,
$\{\widetilde u_i^n\}_{i=0}^{M-1}$, which approximates $u(\cdot, t_n)$ and
want to compute
an approximation for $u(\cdot, t_{n+1})$ at the same nodes, namely,
$\{\widetilde u_i^{n+1}\}_{i=0}^{M-1}$.
To approximate the first time derivative, $u_t=-f(u)_x$,
we use the Shu-Osher finite difference scheme \cite{ShuOsher1989}
with upwinded WENO spatial
reconstructions  \cite{JiangShu96} of order $2r-1$ in the flux function by
\begin{equation}\label{eq:40}
u_{i,n}^{(1)}=u_t(x_i,t_n)=-[f(u)]_x(x_i,t_n)
=-\frac{\hat{f}_{i+\frac{1}{2}}^n-\hat{f}_{i-\frac{1}{2}}^n}{h}+\mathcal{O}(h^{2r-1}),
\end{equation}
with
$$\hat{f}_{i+\mig}^n=\hat{f}(\widetilde u_{i-r+1}^n,\ldots,\widetilde
u_{i+r}^n)$$
being the WENO numerical fluxes of order $2r-1$, with
$r=\lceil\frac{R+1}{2}\rceil$, where $\lceil z\rceil$ denotes the
smallest integer larger or equal than $z$.

Much cheaper centered differences are used instead for the higher
order derivatives.
We expound the general procedure for a third order accurate scheme
($R=3$) for a scalar  one-dimensional conservation law.

First, we compute an approximation of $u_t$ by the procedure stated
above: 
$$\widetilde
u_{i,n}^{(1)}=-\frac{\hat{f}_{i+\mig}^n-\hat{f}_{i-\mig}^n}{h}.$$
We then compute
$$u_{tt}=[u_t]_t=[-f(u)_x]_t=-[f(u)_t]_x=-[f'(u)u_t]_x,$$
where $f'(u)u_t$ is now an approximately known expression for the
required nodes. We 
use then a second order centered difference in order to obtain the
approximation:
$$\widetilde u_{i,n}^{(2)}=-\frac{\widetilde
  f^{(1)}_{i+1,n}-\widetilde f^{(1)}_{i-1,n}}{2h},$$
where
$$\widetilde
f^{(1)}_{i,n}=F_1(\widetilde u_{i,n}^{(0)},\widetilde
u_{i,n}^{(1)})=f'(\widetilde u_{i,n}^{(0)})\widetilde u_{i,n}^{(1)},$$

Finally, we approximate the third derivative:
$$u_{ttt}=[u_t]_{tt}=[-f(u)_x]_{tt}=-[f(u)_{tt}]_x
=-\Big(f''(u)u_t^2+f'(u)u_{tt}\Big)_x,$$
where again the function $f''(u)u_t^2+f'(u)u_{tt}$ is approximately
known at the nodes and therefore $u_{ttt}$ can be
computed by second order accurate centered differences (note that 
in this case it would be required only a first order accurate
approximation; however, the order of centered approximations is always
even):
$$\widetilde u_{i,n}^{(3)}=-\frac{\widetilde
  f^{(2)}_{i+1,n}-\widetilde f^{(2)}_{i-1,n}}{2h},$$
where
$$\widetilde f^{(2)}_{i,n}=F_2(\widetilde u_{i,n}^{(0)},\widetilde
u_{i,n}^{(1)},\widetilde u_{i,n}^{(2)})
=f''(\widetilde u_{i,n}^{(0)})\cdot(\widetilde
u_{i,n}^{(1)})^2+f'(\widetilde u_{i,n}^{(0)})\cdot(\widetilde
u_{i,n}^{(2)})^2.$$
Once all the needed data is obtained, we advance in time by replacing
the terms of the
third order Taylor expansion in time of $u(\cdot, t_{n+1})$ by their
corresponding nodal approximations:
\begin{equation*}
  \begin{split}
    \widetilde u_i^{n+1}=\widetilde u_i^n+\Delta t\widetilde u_{i,n}^{(1)}+\frac{\Delta
      t^2}{2}\widetilde u_{i,n}^{(2)}+\frac{\Delta t^3}{6}\widetilde
    u_{i,n}^{(3)}.
  \end{split}
\end{equation*}

As we shall see, the above example can be extended to arbitrarily
high order time 
schemes through the computation of the suitable high order central
differences of the nodal values
$$\widetilde f^{(l)}_{i,n}=F_l(\widetilde u_{i,n}^{(0)},\widetilde
u_{i,n}^{(1)},\ldots,\widetilde u_{i,n}^{(l)})
=f_{i,n}^{(l)}+\mathcal{O}(h^{R-l}).$$
The generalization to multiple dimensions is straightforward, since 
now the Cauchy-Kovalevskaya procedure, being based on the fact that
$u_t=-\nabla\cdot f(u)$, yields
\begin{equation}\label{eq:cauchy-kovalevskaya}
\frac{\partial^l u}{\partial t^l}=-\nabla \cdot
\Big(\frac{\partial^{l-1}f(u) }{\partial t^{l-1}}\Big)
=-\sum_{i=1}^{d}\frac{\partial}{\partial x_i}\left(\frac{\partial^{l-1}f^{i}(u) }{\partial t^{l-1}}\right)
\end{equation}
and that the spatial reconstruction procedures are done separately for
each dimension. As 
for the case of the systems of equations, the time derivatives are now
computed through tensor products of the corresponding derivatives
of the Jacobian of the fluxes. The following result
\cite{faadibruno1857}, based on the previous works
\cite{TiburceAbadie1850,TiburceAbadie1852}, describes a procedure to compute
them through a generalization of the chain rule.  This formula has been already used 
in the context of numerical analysis, see e.g. \cite{Hairer1993}, \cite{Hickernell2008}, \cite{You2014}.

\begin{theorem}[Faà di Bruno's formula]\label{th:faadibruno}
  Let $f:\mathbb{R}^m\rightarrow\mathbb{R}^{p},$
  $u:\mathbb{R}\rightarrow\mathbb{R}^m$ $n$ times continuously
  differentiable. Then
\begin{equation}\label{eq:faadibruno}
\frac{d^nf(u(t))}{dt^n}=\sum_{s\in\mathcal{P}_n}\left(\begin{array}{c}
      n \\
      s
      \end{array}\right)f^{(|s|)}(u(t))D^su(t),
  \end{equation}    
  where $\mathcal{P}_{n}=\{ s\in\mathbb N^{n} /
\sum_{j=1}^{n} j s_j=n \}$, $|s|=\sum_{j=1}^{n} s_j$,
$\displaystyle\left(\begin{array}{c}
      n \\
      s
      \end{array}\right)=\frac{n!}{s_1!\cdots s_n!}$, 
    $D^s u(t)$ is an $m\times |s|$ matrix whose
    ($\sum\limits_{l<j}s_l+i$)-th column is given by
\begin{equation}\label{eq:ds}
  \displaystyle (D^s u(t))_{\sum\limits_{l<j}s_l+i}=\frac{1}{j!}\frac{\partial^{j}
      u(x)}{\partial t^j},\\
    \quad i=1,\dots,s_j,\quad j=1,\dots,n,
  \end{equation}
  and the action of the $k$-th derivative tensor of $f$ on a
    $m\times k$ matrix $A$ is  given by 
  \begin{equation}\label{eq:77}
      f^{(k)}(u)A=\sum_{i_1,\dots,i_k=1}^{m}\frac{\partial^k f}{\partial
    u_{i_1}\dots\partial u_{i_k}}(u) A_{i_1,1}\dots A_{i_{k},k}\in\mathbb{R}^{p}.
\end{equation}
\end{theorem}
The general procedure for systems and multiple
dimensions is thus easily generalizable and further details about the
procedure can be found in \cite{QiuShu2003}.

\section{The approximate Lax-Wendroff-type procedure}\label{cka}

As reported by the authors of \cite{QiuShu2003}, the computation of
the exact nodal values of $f^{(k)}$ can be very
expensive as $k$ increases, since the number of  required operations
increases exponentially. Moreover, implementing it is costly and requires
large symbolic computations for each equation.

We now present an alternative, which is much less expensive for large $k$
and agnostic about the equation, in the sense that its only
requirement is the knowledge of the flux function.
This procedure also works in the multidimensional case and in the case
of systems as well (by
working componentwise). The rationale of our proposal relies on the
fact that the exact computation of $\frac{\partial^{l-1}f(u)
}{\partial t^{l-1}}$ in \eqref{eq:cauchy-kovalevskaya} from Fa\`a di
Bruno's formula \eqref{eq:faadibruno} requires the knowledge of the
partial derivatives $\frac{\partial^k f}{\partial
    u_{i_1}\dots\partial u_{i_k}}(u)$, $k=1,\dots,l-1$, which entails
  ad-hoc symbolic calculations, but in \eqref{eq:approximations} only
  approximations of $\frac{\partial^{l-1}f(u)
}{\partial t^{l-1}}$  are required, and these can be obtained by
suitable approximated  Taylor expansions in time and 
centered finite differences for the adequate approximation of high
order derivatives.

\subsection{Scheme formulation}

We next introduce some notation which will help in the description of
the approximate fluxes technique along this section. We assume a
 one-dimensional system for the sake of simplicity.

For a function $u\colon \mathbb R\to \mathbb R^{m}$, we denote the 
function on  the grid defined by a base
  point $a$ and grid space $h$ by
  \begin{equation*}
    G_{a,h}(u)\colon\mathbb{Z} \to \mathbb {R}^{m},\quad
    G_{a,h}(u)_i=u(a+ih).
  \end{equation*}
We denote by  $ \Delta^{p,q}_{h}$
the centered finite differences operator that  approximates $p$-th order
  derivatives to order $2q$ on grids with spacing $h$. For any $u$
  sufficiently differentiable, it satisfies:
  \begin{align}\label{eq:3}
    \Delta^{p,q}_{h}
    G_{a,h}(u)=u^{(p)}(a)+\alpha^{p,q}u^{(p+2q)}(a)h^{2q}+\bigO(h^{2q+2}),
  \end{align}
  see Proposition \ref{even} for more details.

We aim to define approximations $\widetilde u^{(k)}_{i,n} \approx
u^{(k)}_{i,n}$, $k=0,\dots,R$, recursively. We start the recursion with
\begin{equation}\label{eq:151}
  \begin{aligned}
    \widetilde u^{(0)}_{i,n}&=u_{i}^{n}\\
    \widetilde
    u^{(1)}_{i,n}&=-\frac{\hat{f}_{i+\frac{1}{2}}^n-\hat{f}_{i-\frac{1}{2}}^n}{h},
  \end{aligned}
\end{equation}   
  where $\hat f_{i+\mig}^{n}$ are computed by applying upwind WENO
  reconstructions to split fluxes obtained from the data $(u_{i}^{n})$
  at time step $n$ (see
  \cite{ShuOsher1989,DonatMarquina96,JiangShu96} for further details)

Associated to fixed $h, i, n$, once obtained $\widetilde
u^{(l)}_{i,n}$, $l=0,\dots,k$, in the recursive process, we define
the $k$-th degree approximated Taylor polynomial $T_k[h,i, n]$ by
\begin{align}\label{eq:taylor}
  T_k[h,i, n](\rho)=\sum_{l=0}^{k}\frac{\widetilde{u}^{(l)}_{i,n} }{l!} \rho^l.
\end{align}

By recursion,  for $k=1,\dots,R-1$, we define
\begin{equation}\label{eq:15}
  \begin{aligned}
    \widetilde f^{(k)}_{i,n} &=
    \Delta_{\delta}^{k,\left\lceil \frac{R-k}{2}\right\rceil}\Big(G_{0,\delta}\big(f(T_k[h, i,
    n])\big)\Big)\\
    \widetilde
    u^{(k+1)}_{i,n}&=-\Delta_h^{1,\left\lceil\frac{R-k}{2}\right\rceil}
    \widetilde f^{(k)}_{i+\cdot,n},
  \end{aligned}
\end{equation}   
where we denote by $\widetilde
f^{(k)}_{i+\cdot,n}$ the vector 
given by the elements $(\widetilde f^{(k)}_{i+\cdot,n})_{j}=\widetilde
f^{(k)}_{i+j,n}$,  recall that $\delta=\Delta t $ and, as previously mentioned,
$\widetilde u_{i,n}^{(0)}=u_{i}^{n}$ is 
the data at the $n$-th time step. With all these ingredients, the
proposed scheme is:
\begin{equation}\label{eq:60}
u_i^{n+1}=u_{i}^{n}+\sum_{l=1}^R\frac{\Delta
  t^l}{l!}\widetilde u_{i,n}^{(l)}.
\end{equation}

The algorithm for the approximated Lax-Wendroff-type procedure can
  be summarized as follows:

\begin{enumerate}
\item Given the numerical data from the time step $t_n$, $\widetilde
  u_{i,n}^{(0)}=u_{i}^{n}$, compute $$\widetilde
  u_{i,n}^{(1)}=-\frac{\hat{f}_{i+\frac{1}{2},n}-\hat{f}_{i-\frac{1}{2},n}}{h}$$
  through upwind spatial WENO reconstructions.
\item 
For $1\leq k \leq R-1$, compute $\widetilde u^{(k+1)}_{i,n}$ from the known data
$\widetilde u^{(l)}_{i,n}, 0\leq l \leq k$, through the following steps
  \begin{enumerate}
  \item Approximate the $k$-th time derivative of $f(u)$ is
     using central differences, as
    indicated in Equation
    (\ref{eq:15}): $$\widetilde f^{(k)}_{i,n} =
    \Delta_{\delta}^{k,\left\lceil \frac{R-k}{2}\right\rceil}\Big(G_{0,\delta}\big(f(T_k[h, i,
    n])\big)\Big),$$
    with
    $$T_k[h,i, n](\rho)=\sum_{l=0}^{k}\frac{\widetilde{u}^{(l)}_{i,n}
    }{l!} \rho^l.$$
  \item Compute nodal approximations of the $(k+1)$-th time derivative of $u$ at $t_n$
    through the suitable central differences,
    according again to Equation (\ref{eq:15}):
    $$\widetilde u^{(k+1)}_{i,n}=-\Delta_h^{1,\left\lceil\frac{R-k}{2}\right\rceil}
    \widetilde f^{(k)}_{i+\cdot,n}.$$
  \end{enumerate}
\item Compute the updated solution at time $t_{n+1}$ by \eqref{eq:60}:
  $$u_i^{n+1}=u_{i}^{n}+\sum_{l=1}^R\frac{\Delta
  t^l}{l!}\widetilde u_{i,n}^{(l)}.$$
\end{enumerate}

\subsection{Fifth order scheme example}

For the sake of illustration, we next detail the specific
numerical scheme that we use for the numerical
experiments included in this paper, detailing how the
aforementioned recursive procedure is performed. In order to simplify as most
as possible the notation, we only show it for the scalar 1D case, as
in the case of systems it consists on working through components and
the multidimensional case working through each line and respective
flux, yielding a rather
simple generalization. We use a fifth
order accurate in space scheme ($r=3$), with fifth order accurate
time discretizations ($R=5$), yielding a fifth order accurate scheme.

The scheme to obtain the upwinded approximation of the first
derivative, $u^{(1)}$, is based on the Shu-Osher finite differences of cell
interfaces,
$$\widetilde u_{i,n}^{(1)}=-\frac{\hat{f}_{i+\frac{1}{2},n}-\hat{f}_{i-\frac{1}{2},n}}{h},\quad0\leq
i<M,$$
where each interface value $\hat{f}_{i-\frac{1}{2},n},$ $0\leq i<M$,
is computed through upwinded fifth order WENO reconstructions. In order to
obtain the last three approximations from both corners we need three
additional ghost cell values at each side of the
stencil, $u_{-3,n},u_{-2,n},u_{-1,n}$ and
$u_{M,n},u_{M+1,n},u_{M+2,n}$, which are obtained through the suitable
numerical boundary conditions, involving the computational domain and
the boundary conditions themselves, if any.

Below it is expounded how to obtain the higher order derivatives
through the flux approximation procedure. 

We compute nodal approximations of the second order time
derivative of $u$, $\widetilde u^{(2)}$, by performing the following
operation:
\begin{equation*}
  \begin{split}
    \widetilde
    f_{i,n}^{(1)}&=\Delta_{\delta}^{1,\left\lceil
        \frac{5-1}{2}\right\rceil}(G_{0,\delta}\big(f(T_1[h,i,n]))\big)=\Delta_{\delta}^{1,2}(G_{0,\delta}\big(f(T_1[h,i,n]))\big)\\
    &=\frac{\varphi_{i,n}^1(2\delta)-8\varphi_{i,n}^1(\delta)
      +8\varphi_{i,n}^1(-\delta)-\varphi_{i,n}^1(-2\delta)}{12\delta},\quad
-2\leq i<M+2,
  \end{split}
\end{equation*}
where
$$\varphi_{i,n}^1(\rho)=f(\widetilde u_i^n+\rho \widetilde u^{(1)}_{i,n}).$$
We need thus to apply previously boundary conditions in order to
obtain two ghost cell values at both sides,
$\widetilde u_{-2}^{(1)},\widetilde u_{-1}^{(1)}$ and $\widetilde
u_{M}^{(1)},\widetilde u_{M+1}^{(1)}$.

We can now define the approximation of the second time derivative of
$u$ as the following fourth order accurate central divided difference
$$\widetilde u^{(2)}_i=-\frac{\widetilde f_{i-2,n}^{(1)}-8\widetilde
  f_{i-1,n}^{(1)}+8\widetilde f_{i+1,n}^{(1)}-\widetilde f_{i+2,n}^{(1)}}{12h}.$$
The nodal approximations of $u^{(3)}$ are obtained in a similar
fashion:
\begin{equation*}
  \begin{split}
    \widetilde
    f_{i,n}^{(2)}&=\Delta_{\delta}^{2,\left\lceil
        \frac{5-2}{2}\right\rceil}(G_{0,\delta}\big(f(T_2[h,i,n]))\big)
    =\Delta_{\delta}^{2,2}(G_{0,\delta}\big(f(T_2[h,i,n]))\big)\\
    &=\frac{-\varphi_{i,n}^2(2\delta)+16\varphi_{i,n}^2(\delta)
      -30\varphi_{i,n}^2(0)
      +16\varphi_{i,n}^2(-\delta)-\varphi_{i,n}^2(-2\delta)}{12\delta^2},
    \\&-2\leq i<M+2,
  \end{split}
\end{equation*}
where
$$\varphi_{i,n}^2(\rho)=f(\widetilde u_i^n+\rho \widetilde
u^{(1)}_{i,n}+\frac{\rho^2}{2}\widetilde u^{(2)}_{i,n}),$$
by previously having computed through boundary conditions
$\widetilde u_{-2}^{(2)},\widetilde u_{-1}^{(2)}$ and $\widetilde
u_{M}^{(2)},\widetilde u_{M+1}^{(2)}$
and, again, use a fourth order central divided difference to
approximate the third time derivative of $u$:
$$\widetilde u^{(3)}_i=-\frac{\widetilde f_{i-2}^{(2)}-8\widetilde
  f_{i-1}^{(2)}+8\widetilde f_{i+1}^{(2)}-\widetilde f_{i+2}^{(2)}}{12h}.$$
Since for the fourth and fifth time derivative of $u$ it is only required
approximations of second and first order, respectively, we approximate
the corresponding time derivatives of the flux through second order
central differences and perform as well second order central
differences between them.

On the one hand, we have
\begin{equation*}
  \begin{split}
    \widetilde
    f_{i,n}^{(3)}&=\Delta_{\delta}^{3,\left\lceil
        \frac{5-3}{2}\right\rceil}(G_{0,\delta}\big(f(T_3[h,i,n]))\big)
    =\Delta_{\delta}^{3,1}(G_{0,\delta}\big(f(T_3[h,i,n]))\big)\\
    &=\frac{\varphi_{i,n}^3(2\delta)-2\varphi_{i,n}^3(\delta)
      +2\varphi_{i,n}^3(-\delta)-\varphi_{i,n}^3(-2\delta)}{2\delta^3},\quad
-1\leq i<M+1,
  \end{split}
\end{equation*}
with
$$\varphi_{i,n}^3(\rho)=f(\widetilde u_i^n+\rho \widetilde
u^{(1)}_{i,n}+\frac{\rho^2}{2}\widetilde
u^{(2)}_{i,n}+\frac{\rho^3}{6}\widetilde u^{(3)}_{i,n}),$$
where $\widetilde u^{(3)}_{-1}$ and $\widetilde u^{(3)}_{M}$ have been
computed using the adequate numerical boundary conditions.

Then we define
$$\widetilde u^{(4)}_{i,n}=-\frac{\widetilde f^{(3)}_{i+1,n}-\widetilde
  f^{(3)}_{i-1,n}}{2h}.$$
On the other hand,
\begin{equation*}
  \begin{split}
    \widetilde
    f_{i,n}^{(4)}&=\Delta_{\delta}^{4,\left\lceil
        \frac{5-4}{2}\right\rceil}(G_{0,\delta}\big(f(T_4[h,i,n]))\big)=\Delta_{\delta}^{4,1}(G_{0,\delta}\big(f(T_4[h,i,n]))\big)\\
    &=\frac{\varphi_{i,n}^4(2\delta)-4\varphi_{i,n}^4(\delta)
      +6\varphi_{i,n}^4(0)
      -4\varphi_{i,n}^4(-\delta)+\varphi_{i,n}^4(-2\delta)}{\delta^4},\quad
-1\leq i<M+1,
  \end{split}
\end{equation*}
with
$$\varphi_{i,n}^4(\rho)=f(\widetilde u_i^n+\rho \widetilde
u^{(1)}_{i,n}+\frac{\rho^2}{2}\widetilde
u^{(2)}_{i,n}+\frac{\rho^3}{6}\widetilde
u^{(3)}_{i,n}+\frac{\rho^4}{24}\widetilde u^{(4)}_{i,n}),$$
where $\widetilde u^{(4)}_{-1}$ and $\widetilde u^{(4)}_{M}$ are
obtained through numerical boundary conditions.

We then define
$$\widetilde u^{(5)}_{i,n}=-\frac{\widetilde
  f^{(4)}_{i+1,n}-\widetilde f^{(4)}_{i-1,n}}{2h}.$$
The next time step is then computed through the fifth order Taylor
expansion replacing the derivatives with their corresponding
approximations:
$$\widetilde u_i^{n+1}=\widetilde u_i^n+\Delta t\widetilde
u_{i,n}^{(1)}+\frac{\Delta t^2}{2}\widetilde
u_{i,n}^{(2)}+\frac{\Delta t^3}{6}\widetilde u_{i,n}^{(3)}+
\frac{\Delta t^4}{24}\widetilde u_{i,n}^{(4)}+\frac{\Delta
  t^5}{120}\widetilde u_{i,n}^{(5)}.$$

\subsection{Theoretical results}

The next result is the main result of this paper. Its proof is
deferred to appendix \ref{appendixb}.

\begin{proposition}\label{prop:1}
  The scheme defined by \eqref{eq:15} and \eqref{eq:60} is $R$-th
  order accurate. 
\end{proposition}

The following result yields optimal central finite difference
approximations to derivatives of any order. We denote by $\lfloor z
\rfloor$ the integer part of  $z$.

\begin{proposition}\label{even}
  For any  $p, q\in\mathbb{N}$, there exist  $\beta_{l}^{p,q}$,
  $l=0,\dots,s:=\lfloor \frac{p-1}{2}\rfloor+q$ such that
  \begin{align}\label{eq:3bis}
    \Delta^{p,q}_{h} v=\frac{1}{h^p}\sum_{l=0}^{s}\beta_{l}^{p,q}(v_l+(-1)^pv_{-l})
  \end{align}
  satisfies \eqref{eq:3}.
\end{proposition}

\begin{proof}
We set
\begin{equation}\label{eq:13}
  \Delta_{h}^{p,q}v=\frac{1}{h^p}\sum_{l=-s}^{s}\beta_{l}^{p,q} v_{l},\quad s=\left\lfloor
    \frac{p-1}{2}\right\rfloor +q
\end{equation}
for $\beta_{l}^{p,q}$ to determine such that
\begin{align*}
  \psi(h)=\psi^{p,q}(h)=\sum_{l=-s}^{s}\beta_{l}^{p,q} u(a+lh),
\end{align*}     
satisfies
\begin{equation}\label{eq:8}
  \psi^{(r)}(0)=0,\quad r=0,\dots,2s, r\neq p, \quad \psi^{(p)}(0)=p!u^{(p)}(a).
\end{equation}
Since 
\begin{align*}
  \psi^{(r)}(0)=\sum_{l=-s}^{s}\beta_{l}^{p,q} l^r u^{(r)}(a),
\end{align*}
\eqref{eq:8} is equivalent to  the system of $2s+1$ equations and
$2s+1$ unknowns
\begin{equation}\label{eq:9}
  \begin{aligned}
    \sum_{l=-s}^{s}\beta_{l}^{p,q} l^r &=0,\quad r=0,\dots,2s, r\neq p\\
    \sum_{l=-s}^{s}\beta_{l}^{p,q} l^p &=p!,
  \end{aligned}
\end{equation} 
whose coefficient matrix is a Vandermonde invertible matrix, and  it
thus have a unique solution. We see now that if $p$ is even then
$\beta_{-l}^{p,q} = \beta_{l}^{p,q}$, $l=1,\dots,s$, and if it is odd then
$\beta_{-l}^{p,q} = -\beta_{l}^{p,q}$, $l=0,\dots,s$. For the first
case \eqref{eq:9} yields
\begin{align*}
  \sum_{l=-s}^{s}\beta_{l}^{p,q} l^r &=0,\quad r=1,3,\dots,2s-1\\
  \sum_{l=1}^{s}(\beta_{l}^{p,q}-\beta_{-l}^{p,q}) l^r &=0,\quad
  r=1,3,\dots,2s-1,
\end{align*}    
which is a homogeneous system of $s$ equations with $s$ unknowns,
with an invertible (Vandermonde) matrix, therefore
$\beta_{l}^{p,q}-\beta_{-l}^{p,q}=0$, $l=1,\dots,s$. The case for odd
$p$ is handled similarly.

With this, we have from a Taylor expansion of $\psi$ that:
\begin{align*}
  \Delta_{h}^{p,q}G_{a,h}u&=\frac{1}{h^p}\psi(h)=
  \frac{1}{h^p}(\frac{\psi^{(p)}(0)}{p!}h^p+\sum_{r=2s+1}^{\infty}
  \frac{\psi^{(r)}(0)}{r!} h^r)\\
  &=
  u^{(p)}(0)+\sum_{r=2s+1}^{\infty}
  \sum_{l=-s}^{s}\beta_{l}^{p,q}l^r
  \frac{u^{(r)}(a)}{r!} h^{r-p}\\
  &=
  u^{(p)}(0)+\sum_{r=2s+1}^{\infty}
  \sum_{l=1}^{s}(\beta_{l}^{p,q}+(-1)^r\beta_{-l}^{p,q})l^r
  \frac{u^{(r)}(a)}{r!} h^{r-p}
\end{align*}
Now, if $p$ is even, then $\beta_{l}^{p,q}=\beta_{-l}^{p,q}$ and,
therefore, the only remaining terms are those with even $r$:
\begin{align*}
  &\sum_{r=2s+1}^{\infty}
  \sum_{l=1}^{s}(\beta_{l}^{p,q}+(-1)^r\beta_{-l}^{p,q})l^r
  \frac{u^{(r)}(a)}{r!} h^{r-p} \\
  &= \sum_{m=s+1}^{\infty}  \alpha_{m}^{p,q}u^{(2m)}(a)
  h^{2m-p},\quad 
  \alpha_{m}^{p,q}=  \frac{2}{(2m)!}\sum_{l=1}^{s}\beta_{l}^{p,q}l^{2m}.
\end{align*}
On the other hand, if $p$ is odd, then $\beta_{-l}^{p,q}=-\beta_{l}^{p,q}$ and,
therefore, the only remaining terms are those with odd $r$:
\begin{align*}
  &\sum_{r=2s+1}^{\infty}
  \sum_{l=1}^{s}(\beta_{l}^{p,q}+(-1)^r\beta_{-l}^{p,q})l^r
  \frac{u^{(r)}(a)}{r!} h^{r-p} \\
  &= \sum_{m=s}^{\infty}  \alpha_{m}^{p,q}u^{(2m+1)}(a)
  h^{2m+1-p},\quad 
  \alpha_{m}^{p,q}=  \frac{2}{(2m)!}\sum_{l=1}^{s}\beta_{l}^{p,q}l^{2m+1}.
\end{align*}
One can check that the definition of $s$ gives that the smallest exponent
in the remainder terms is $2q$. The result follows easily if one
redefines $\beta_0^{p,q}=\beta_0^{p,q}/2$ for even $p$ (for odd $p$ it
is 0).
\qed
\end{proof}

Finally, we next present a result which ensures that our scheme, being based
on  approximations of  flux derivatives, is conservative.
\begin{theorem}
  The scheme resulting of the flux approximation procedure can be
  written in conservation form, namely,
  \begin{equation}\label{eq:12}
    u_{i}^{n+1}=   u_{i}^{n}-\frac{\Delta t}{h}\big(
    \hat g_{i+\mig}^{n}-\hat g_{i-\mig}^{n}\big).
  \end{equation}
\end{theorem}

\begin{proof}
The key to \eqref{eq:12} is to express $\Delta_h^{1,q}v$ in
\eqref{eq:13} in a conservative way:
\begin{equation}\label{eq:16}
  \begin{aligned}
    \Delta_h^{1,q}v&:=\frac{1}{h}\sum_{l=-q}^{q}\beta_{l}^{1,q} v_{l} = 
    \frac{1}{h}
    \left(
      \sum_{l=-q}^{q-1}\gamma_{l}^{q} v_{l+1}
      -\sum_{l=-q}^{q-1}\gamma_{l}^{q} v_{l}
    \right),\\
    &=\frac{1}{h}
    \left(
      \gamma_{q-1}^{q}v_{q}+\sum_{l=-q+1}^{q-1}(\gamma_{l-1}^{q}
      -\gamma_{l}^{q} )v_{l}
      -\gamma_{-q}v_{-q}^{q}
    \right),
  \end{aligned}
\end{equation}
with $\gamma_{l}^{q}$ to be determined. Since the latter ought to be
satisfied by any $v$, we deduce that
\begin{align*}
  \gamma_{q-1}^{q}&=\beta_q^{1,q}\\
  \gamma_{l-1}^{q}
  -\gamma_{l}^{q} &=\beta_{l}^{1,q},\quad l=-q+1, \dots q-1.\\
  -\gamma_{-q}^{q}&=\beta_{-q}^{1,q}
\end{align*}
This is a system of $2q+1$ equations with $2q$ unknowns:
\begin{equation}\label{eq:14}
  \begin{bmatrix}
    -1&0&0&\dots&\dots&0\\
    1&-1&0&\dots&\dots&0\\
    0&1&-1&\dots&\dots&0\\
    \vdots&&\ddots&\ddots&&\vdots\\
    \vdots&&&\ddots&\ddots&\vdots\\
    0&\dots&\dots&0&1&-1\\
    0&\dots&&\dots&0&1\\
  \end{bmatrix}
  \begin{bmatrix}
    \gamma^{q}_{-q}\\
    \gamma^{q}_{-q+1}\\
    \vdots\\
    \vdots\\
    \gamma^{q}_{q-2}\\
    \gamma^{q}_{q-1}
  \end{bmatrix}    
  =
  \begin{bmatrix}
    \beta^{1,q}_{-q}\\
    \beta^{1,q}_{-q+1}\\
    \vdots\\
    \vdots\\
    \vdots\\
    \beta^{1,q}_{q-1}\\
    \beta^{1,q}_{q}       
  \end{bmatrix}    
\end{equation}
The subsystem formed by the first $2q$ equations has a lower
triangular invertible matrix, hence the first $2q$ equations can be
uniquely solved. Elimination of the elements in the subdiagonal from
those in the diagonal yields  the determinant of the matrix:
\begin{equation*}
  \det\begin{bmatrix}
    -1&0&0&\dots&\dots&0&            \beta^{1,q}_{-q}\\   
    1&-1&0&\dots&\dots&0&            \beta^{1,q}_{-q+1}\\ 
    0&1&-1&\dots&\dots&0&            \vdots\\             
    \vdots&&\ddots&\ddots&&\vdots&   \vdots\\             
    \vdots&&&\ddots&\ddots&\vdots&   \vdots\\             
    0&\dots&\dots&0&1&-1&            \beta^{1,q}_{q-1}\\  
    0&\dots&&\dots&0&1&              \beta^{1,q}_{q}      
  \end{bmatrix}
  =(-1)^{2q}\sum_{l=-q}^{q}\beta^{1,q}_{l}
\end{equation*}
By \eqref{eq:9} with $r=0$, $\sum_{l=-q}^{q}\beta^{1,q}_{l}=0$,
therefore  system \eqref{eq:14} has a unique solution.

With the notation $\mu=\left\lceil\frac{R-l}{2}\right\rceil$, 
from \eqref{eq:15}, \eqref{eq:60} and \eqref{eq:16} we deduce:
\begin{align*}
  u_{i}^{n+1}&=u_{i}^{n}-\frac{\Delta t}{h}(\hat f_{i+\mig}^n-\hat f_{i-\mig}^n)-\sum_{l=1}^{R-1}\frac{ (\Delta t)^{l+1}}{(l+1)!}
  \Delta_h^{1,\mu}
  \widetilde f^{(l)}_{i+\cdot,n}\\
  &=u_{i}^{n}-\frac{\Delta t}{h}(\hat f_{i+\mig}^n-\hat f_{i-\mig}^n)\\
  &-\sum_{l=1}^{R-1}\frac{ (\Delta t)^{l+1}}{(l+1)!}
  \frac{1}{h}\big(\sum_{s=-\qRl}^{\qRl-1}\gamma_{s}^{\qRl}
  \widetilde f^{(l)}_{i+s+1,n}-\sum_{s=-\qRl}^{\qRl-1}\gamma_{s}^{\qRl}
  \widetilde f^{(l)}_{i+s,n}\big)\\
  &=u_{i}^{n}-\frac{\Delta t}{h}(\hat f_{i+\mig}^n-\hat f_{i-\mig}^n)\\
  &-\frac{\Delta t}{h} \sum_{l=1}^{R-1}\frac{ (\Delta t)^{l}}{(l+1)!}
  \big(\sum_{s=-\qRl}^{\qRl-1}\gamma_{s}^{\qRl}
  \widetilde f^{(l)}_{i+s+1,n}-\sum_{s=-\qRl}^{\qRl-1}\gamma_{s}^{\qRl}
  \widetilde f^{(l)}_{i+s,n}\big)
\end{align*}
and we deduce equation \eqref{eq:12} with
\begin{equation}\label{eq:17}
  \hat g_{i+\mig}^{n}=\hat f_{i+\mig}^n+\sum_{l=1}^{R-1}\frac{
    (\Delta t)^{l}}{(l+1)!} \sum_{s=-\qRl}^{\qRl-1}\gamma_{s}^{\qRl}
  \widetilde f^{(l)}_{i+s+1,n}.
\end{equation}
\qed
\end{proof}

\begin{remark}
  From \eqref{eq:15} and \eqref{eq:3bis} we may deduce that for each
  $i=0,\dots,M-1$, the
  computation   of the coefficients $\widetilde u^{(l)}_{i,n}$, for
  $l=2,\dots,k$, requires $\frac{R^2} {2}+\bigO(R)$ flux evaluations,
  and $\frac{3R^2}{2}+\bigO(R)$ floating point operations. Therefore, the time
  step can be  performed with one extra WENO reconstruction
  for $\widetilde u^{(1)}_{i,n}$ and  about $2R$ more floating point
  operations to evaluate the polynomial in \eqref{eq:60}. 
\end{remark}  

\section{Numerical experiments}\label{nex}

In this section we present some 1D and 2D experiments both for scalar
and system of equations involving comparisons of the fifth order both in
space ($r=3$) and time ($R=2r-1=5$) exact and approximate Lax-Wendroff-type
schemes, together with the results obtained using the third order TVD
Runge-Kutta time discretization. 

For now on we will refer as WENO-LW to the exact Lax-Wendroff-type
procedure, WENO-LWA to the approximate Lax-Wendroff-type procedure and
WENO-RK when a Runge-Kutta discretization is used. In each case,
numbers indicate the spatial accuracy order (first) and
the time accuracy order (second).

\subsection{1D Linear advection equation}
We set as initial condition $u(x,0)=0.25+0.5\sin(\pi x)$, periodic
boundary conditions at both sides, whose exact solution is $u(x,
t)=0.25+0.5\sin(\pi(x-t))$, using both the exact and approximate
Lax-Wendroff-type procedure with fifth order accuracy both in space and
time (WENO5-LW5 and WENO5-LWA5, respectively) and run the
simulation up to $t=1$, with CFL=0.5 (except for RK3, where we set
$\Delta t=h^{\frac{5}{3}}$ in order to achieve fifth order accuracy in time),
and for resolutions $n=20\cdot 2^n$ points, $1\leq n\leq 5$, obtaining
the results shown in Tables \ref{SLA-RK}-\ref{SLA-CKR}.
\begin{table}[htb]
  \centering
  \begin{tabular}{|c|c|c|c|c|}
    \hline
    $n$ & Error $\|\cdot\|_1$ & Order $\|\cdot\|_1$ & Error
    $\|\cdot\|_{\infty}$ & Order $\|\cdot\|_{\infty}$ \\
    \hline
    40 & 1.13E$-5$ & $-$ & 2.39E$-5$ & $-$  \\
    \hline
    80 & 3.49E$-7$ & 5.02 & 7.17E$-7$ & 5.06 \\
    \hline
    160 & 1.09E$-8$ & 5.00 & 2.25E$-8$ & 4.99 \\
    \hline
    320 & 3.41E$-10$ & 5.00 & 6.77E$-10$ & 5.06 \\
    \hline
    640 & 1.15E$-11$ & 4.89 & 2.23E$-11$ & 4.93 \\
    \hline
    1280 & 3.51E$-12$ & 1.71 & 8.32E$-12$ & 1.42 \\
    \hline
  \end{tabular}
  \caption{Error table for linear advection equation, $t=1$. WENO5-RK3.}
  \label{SLA-RK}
\end{table}
\begin{table}[htb]
  \centering
  \begin{tabular}{|c|c|c|c|c|}
    \hline
    $n$ & Error $\|\cdot\|_1$ & Order $\|\cdot\|_1$ & Error
    $\|\cdot\|_{\infty}$ & Order $\|\cdot\|_{\infty}$ \\
    \hline
    40 & 1.09E$-5$ & $-$ & 2.37E$-5$ & $-$  \\
    \hline
    80 & 3.29E$-7$ & 5.05 & 7.00E$-7$ & 5.08 \\
    \hline
    160 & 1.02E$-8$ & 5.01 & 2.21E$-8$ & 4.98 \\
    \hline
    320 & 3.19E$-10$ & 5.00 & 6.65E$-10$ & 5.06 \\
    \hline
    640 & 9.96E$-12$ & 5.00 & 2.02E$-11$ & 5.04 \\
    \hline
    1280 & 3.12E$-13$ & 4.99 & 6.12E$-13$ & 5.04 \\
    \hline
  \end{tabular}
  \caption{Error table for linear advection equation, $t=1$. WENO5-LW5.}
  \label{SLA-CK}
\end{table}
\begin{table}[htb]
  \centering
  \begin{tabular}{|c|c|c|c|c|}
    \hline
    $n$ & Error $\|\cdot\|_1$ & Order $\|\cdot\|_1$ & Error
    $\|\cdot\|_{\infty}$ & Order $\|\cdot\|_{\infty}$ \\
    \hline
    40 & 1.09E$-5$ & $-$ & 2.37E$-5$ & $-$  \\
    \hline
    80 & 3.29E$-7$ & 5.05 & 7.00E$-7$ & 5.08 \\
    \hline
    160 & 1.02E$-8$ & 5.01 & 2.21E$-8$ & 4.98 \\
    \hline
    320 & 3.19E$-10$ & 5.00 & 6.65E$-10$ & 5.06 \\
    \hline
    640 & 9.96E$-12$ & 5.00 & 2.02E$-11$ & 5.04 \\
    \hline
    1280 & 3.12E$-13$ & 5.00 & 6.12E$-13$ & 5.04 \\
    \hline
  \end{tabular}
  \caption{Error table for linear advection equation, $t=1$. WENO5-LWA5.}
  \label{SLA-CKR}
\end{table}

From the results, we can conclude that all the proposed scheme
achieve the fifth order accuracy. We must remark that the loss of
accuracy appreciable in the last row of the RK3 version with $\Delta
t=h^{\frac{5}{3}}$ is due to accumulation of machine errors because of
a major number of required iterations (produced by the time-space
re-scaling performed to achieve the fifth order accuracy). On the
other hand, the results obtained through the approximated scheme,
WENO5-LWA5, are almost identical than those obtained through the
original version, WENO5-LW5, as should be expected, since in this case
(linear flux) both the exact and the approximate formulation yield
theoretically the same results.

\subsection{1D Burgers equation}
We now work with the inviscid Burgers equation, which is given by
$$u_t+f(u)_x=0,\quad f(u)=\frac{1}{2}u^2,\quad x\in(-1, 1).$$
\subsubsection{Smooth solution}
We perform an accuracy test in this equation with the same setup
(initial and boundary conditions as well as the spatial resolutions)
as in the previous example, except that now we set the end time to
$t=0.3$ with CFL=0.5. The results for
WENO5-LW5 and WENO5-LWA5 are presented in Tables \ref{SB-CK}-\ref{SB-CKR}.
\begin{table}[htb]
  \centering
  \begin{tabular}{|c|c|c|c|c|}
    \hline
    $n$ & Error $\|\cdot\|_1$ & Order $\|\cdot\|_1$ & Error
    $\|\cdot\|_{\infty}$ & Order $\|\cdot\|_{\infty}$ \\
    \hline
    40 & 2.38E$-5$ & $-$ & 2.09E$-4$ & $-$  \\
    \hline
    80 & 7.93E$-7$ & 4.91 & 9.44E$-6$ & 4.47 \\
    \hline
    160 & 2.45E$-8$ & 5.01 & 3.01E$-7$ & 4.97 \\
    \hline
    320 & 7.48E$-10$ & 5.04 & 9.13E$-9$ & 5.05 \\
    \hline
    640 & 2.32E$-11$ & 5.01 & 2.81E$-10$ & 5.02 \\
    \hline
    1280 & 7.22E$-13$ & 5.00 & 8.69E$-12$ & 5.01  \\
    \hline
  \end{tabular}
  \caption{Error table for Burgers equation, $t=0.3$. WENO5-LW5.}
  \label{SB-CK}
\end{table}

\begin{table}[htb]
  \centering
  \begin{tabular}{|c|c|c|c|c|}
    \hline
    $n$ & Error $\|\cdot\|_1$ & Order $\|\cdot\|_1$ & Error
    $\|\cdot\|_{\infty}$ & Order $\|\cdot\|_{\infty}$ \\
    \hline
    40 & 2.38E$-5$ & $-$ & 2.09E$-4$ & $-$  \\
    \hline
    80 & 7.94E$-7$ & 4.91 & 9.46E$-6$ & 4.47 \\
    \hline
    160 & 2.46E$-8$ & 5.01 & 3.02E$-7$ & 4.97 \\
    \hline
    320 & 7.50E$-10$ & 5.04 & 9.15E$-9$ & 5.05 \\
    \hline
    640 & 2.32E$-11$ & 5.01 & 2.81E$-10$ & 5.02 \\
    \hline
    1280 & 7.23E$-13$ & 5.00 & 8.71E$-12$ & 5.01  \\
    \hline
  \end{tabular}
  \caption{Error table for Burgers equation, $t=0.3$. WENO5-LW5.}
  \label{SB-CKR}
\end{table}

In this case, we can see again that the fifth order accuracy is
achieved and the errors both in $\|\cdot\|_1$ and $\|\cdot\|_{\infty}$
of the exact and approximate version are very close.

\subsubsection{Discontinuous solution}
If we now change the final time to $t=12$, the wave breaks at
$t=1.1$ and a shock is then formed. We compare the WENO5-LW5 and
WENO5-LWA5 techniques with WENO5-RK3, whose results are shown in Figure
\ref{db}. We run this simulation using a resolution of $n=80$ points.
\begin{figure}[htb]
\centering
\begin{tabular}{cc}
\multicolumn{2}{c}{\includegraphics[width=0.7\textwidth]{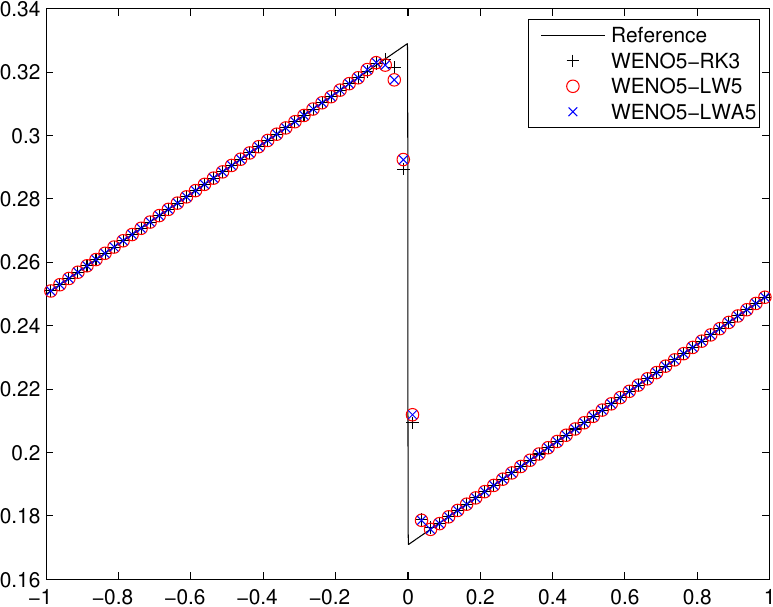}}\\
\multicolumn{2}{c}{(a) Global results.}\\[1ex]
\includegraphics[width=0.45\textwidth]{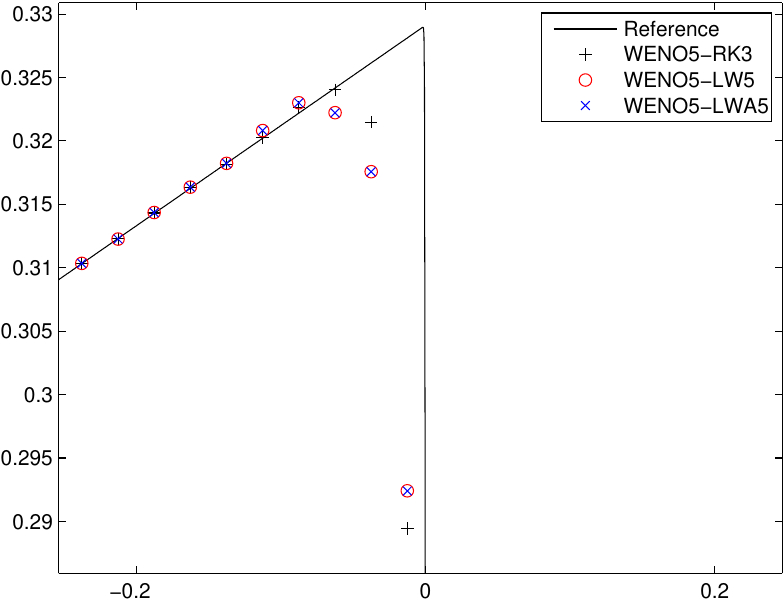} &
\includegraphics[width=0.45\textwidth]{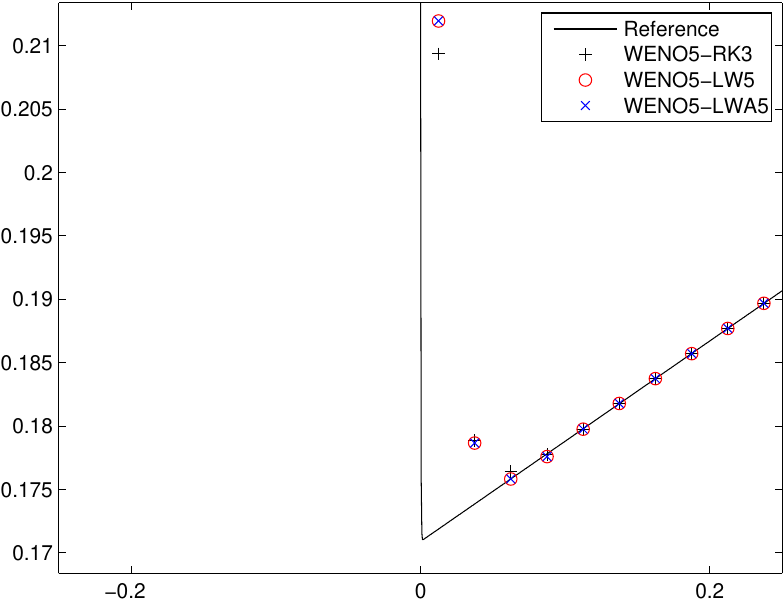} \\
(b) Enlarged view. & (c) Enlarged view.
\end{tabular}
\caption{Discontinuous solution for Burgers equation,
  $t=12$.}
\label{db}
\end{figure}

One can conclude from the results shown in Figure \ref{db} that even
in the discontinuous case the approximate formulation results are
quite close to those obtained through the exact version.

\subsection{1D Euler equations}
The 1D Euler equations for gas dynamics are:

\begin{equation}\label{eq:eulereq1}
\begin{aligned}
&u_t+f(u)_x=0,\quad u=u(x,t),\quad\Omega=(0,1),\\
&u=\left[\begin{array}{c}
    \rho \\
    \rho v \\
    E \\
  \end{array}\right],\hspace{0.3cm}f(u)=\left[\begin{array}{c}
    \rho v \\
    p+\rho v^2 \\
    v(E+p) \\
  \end{array}\right],
\end{aligned}
\end{equation}
where $\rho$ is the density, $v$ is the
velocity and  $E$ is the specific energy of the system. The variable
$p$ stands for the pressure and is given by the equation of state:
$$p=\left(\gamma-1\right)\left(E-\frac{1}{2}\rho v^2\right),$$
where $\gamma$ is the adiabatic constant, that will be taken as
  $1.4$.
\subsubsection{Smooth solution}
We set as initial conditions
$$\left\{\begin{array}{rcl}
    \rho(x, t)&=&0.75+0.5\sin(\pi x) \\
    \rho v(x, t)&=&0.25+0.5\sin(\pi x) \\
    E(x, t)&=&0.75+0.5\sin(\pi x)
    \end{array}\right. ,\quad x\in(-1,1),$$
and periodic boundary conditions for all the quantities. For $t=0.1$
the characteristic lines do not cross so that
the solution remains smooth. We compute a reference solution at that
time with a very fine mesh and perform an order test with WENO5-LW5
and WENO5-LWA5 using
CFL=0.5 for the same spatial resolutions
as the previous examples.

Both the errors and order quantities presented in the tables are an
average of each individual quantity obtained for the three
quantities of $u$. The obtained results are presented in Tables
\ref{SE-CK}-\ref{SE-CKR}.
\begin{table}[htb]
  \centering
  \begin{tabular}{|c|c|c|c|c|}
    \hline
    $n$ & Error $\|\cdot\|_1$ & Order $\|\cdot\|_1$ & Error
    $\|\cdot\|_{\infty}$ & Order $\|\cdot\|_{\infty}$ \\
    \hline
    40 & 2.98E$-4$ & $-$ & 4.70E$-3$ & $-$  \\
    \hline
    80 & 3.36E$-5$ & 3.15 & 5.49E$-4$ & 3.10 \\
    \hline
    160 & 1.60E$-6$ & 4.39 & 4.59E$-5$ & 3.58 \\
    \hline
    320 & 5.53E$-8$ & 4.85 & 1.78E$-6$ & 4.69 \\
    \hline
    640 & 1.76E$-9$ & 4.98 & 6.01E$-8$ & 4.89 \\
    \hline
    1280 & 5.65E$-11$ & 4.96 & 1.84E$-9$ & 5.03 \\
    \hline
  \end{tabular}
  \caption{Error table for 1D Euler equation, $t=0.1$. WENO5-LW5.}
  \label{SE-CK}
\end{table}
\begin{table}[htb]
  \centering
  \begin{tabular}{|c|c|c|c|c|}
    \hline
    $n$ & Error $\|\cdot\|_1$ & Order $\|\cdot\|_1$ & Error
    $\|\cdot\|_{\infty}$ & Order $\|\cdot\|_{\infty}$ \\
    \hline
    40 & 2.98E$-4$ & $-$ & 4.70E$-3$ & $-$  \\
    \hline
    80 & 3.36E$-5$ & 3.15 & 5.49E$-4$ & 3.10 \\
    \hline
    160 & 1.60E$-6$ & 4.39 & 4.59E$-5$ & 3.58 \\
    \hline
    320 & 5.53E$-8$ & 4.85 & 1.78E$-6$ & 4.69 \\
    \hline
    640 & 1.76E$-9$ & 4.98 & 6.01E$-8$ & 4.89 \\
    \hline
    1280 & 5.65E$-11$ & 4.96 & 1.84E$-9$ & 5.03 \\
    \hline
  \end{tabular}
  \caption{Error table for 1D Euler equation, $t=0.1$. WENO5-LWA5.}
  \label{SE-CKR}
\end{table}

In this case it can be seen that the results shown in the tables are
identical at the accuracy in which the errors have been
displayed. This again indicates that the approximate version provides
essentially the same results than the exact version. For instance, the
global average error ($\|\cdot\|_1$) including all three components
and cells of the numerical solution at $t=0.1$ at the resolution
$n=1280$ of the approximate flux approach with respect to the exact
flux approach is 2.82E$-15$.

\subsubsection{Shu-Osher problem}
We now consider the interaction with a Mach 3 shock and a sine wave. The
spatial domain is now given by $\Omega:=(-5,5)$, with initial
conditions
$$\begin{cases}
  \left.\begin{array}{rcl}
    \rho(x, t)&=&3.857143 \\
    v(x, t)&=&2.629369 \\
    p(x, t)&=&10.33333
    \end{array}\right\} & \mbox{if } x\leq-4 \\
  \left.\begin{array}{rcl}
    \rho(x, t)&=&1.0+0.2\sin(5x) \\
    v(x, t)&=&0 \\
    p(x, t)&=&1
    \end{array}\right\} & \mbox{if } x>-4 \\
  \end{cases}$$
with left inflow and right outflow boundary conditions.

We run one simulation until $t=1.8$ and compare the results obtained
with WENO5-RK3, WENO5-LW5 and WENO5-LWA5, $n=400$ cells,
$\textnormal{CFL}=0.5$ with a reference solution computed with 16000
grid points. The results are shown in Figure \ref{so}.
\begin{figure}[htb]
    \centering
  \begin{tabular}{cc}
    \includegraphics[width=0.45\textwidth]{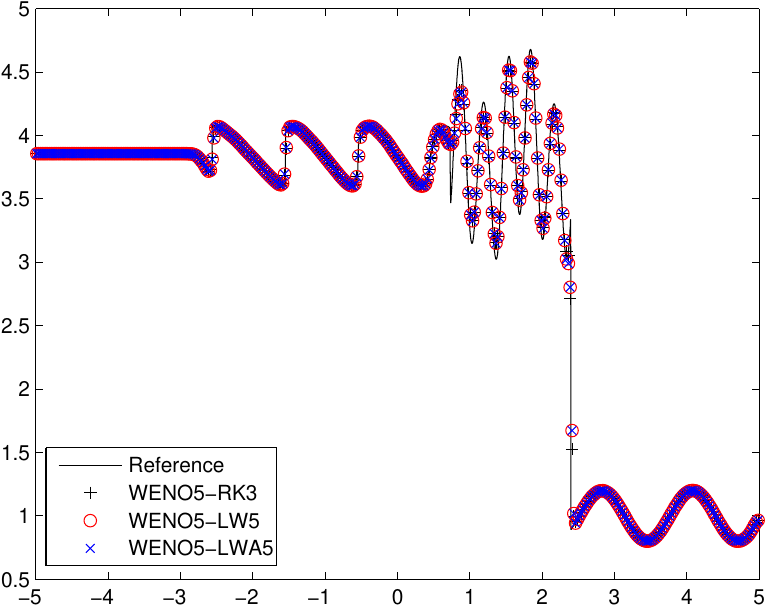}
    & \includegraphics[width=0.45\textwidth]{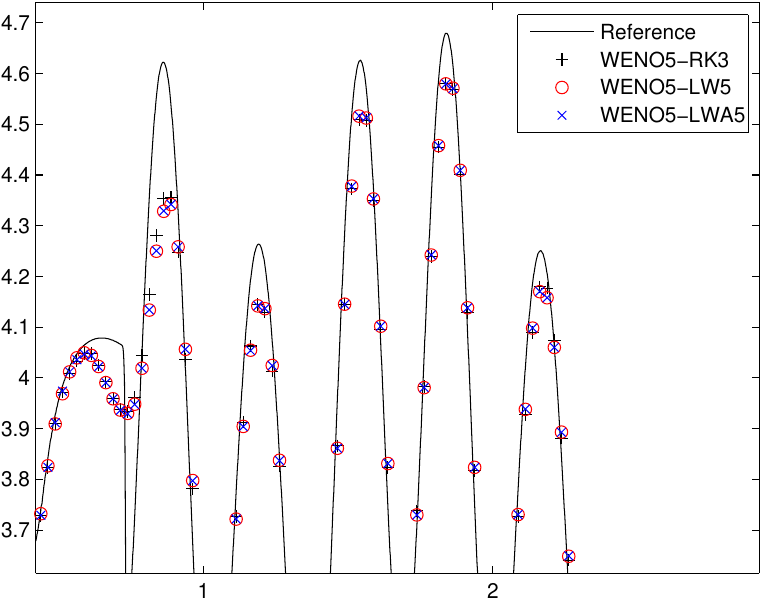} \\
    (a) Global results & (b) Enlarged view \\
    \includegraphics[width=0.45\textwidth]{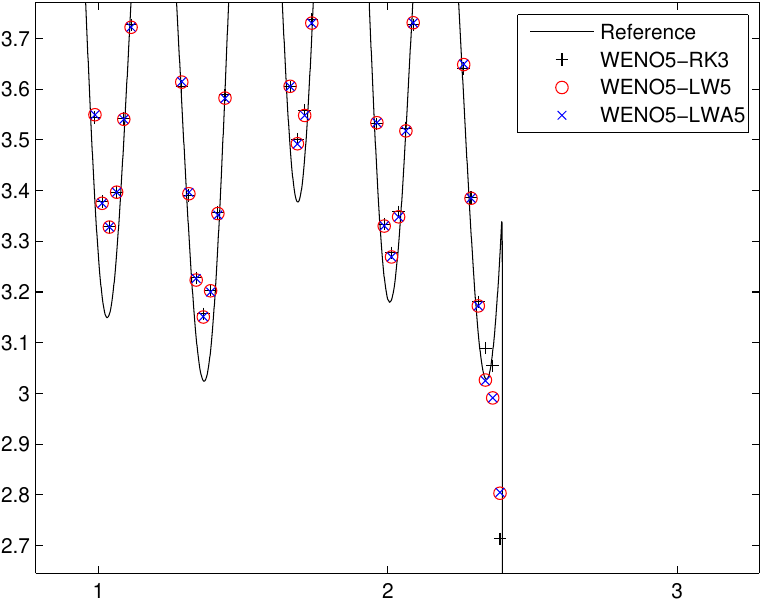}
    & \includegraphics[width=0.45\textwidth]{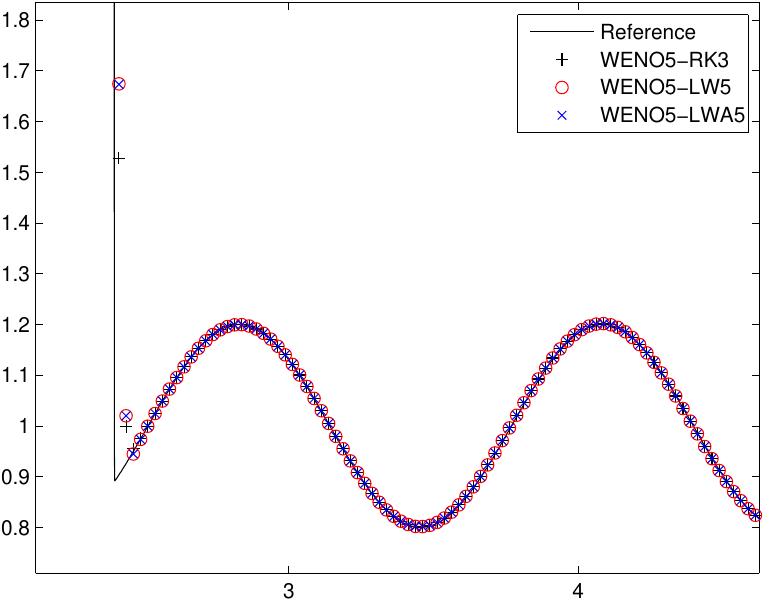} \\
    (c) Enlarged view & (d) Enlarged view \\
  \end{tabular}
  \caption{Shu-Osher problem. Pressure field.}
  \label{so}
\end{figure}

From the results it can be concluded that again the behavior of both
Lax-Wendroff-type techniques is almost the same.

\subsubsection{Blast wave}
Now the initial data is the following one, corresponding to the 
interaction of two blast waves:
$$u(x,0)=\left\{\begin{array}{ll}
    u_L & 0<x<0.1,\\
    u_M & 0.1<x<0.9,\\
    u_R & 0.9<x<1,
    \end{array}\right.$$
where $\rho_L=\rho_M=\rho_R=1$, $v_L=v_M=v_R=0$,
$p_L=10^3,p_M=10^{-2},p_R=10^2$. Reflecting boundary conditions are set
at $x=0$ and $x=1$, simulating a solid wall at both sides. This
problem involves multiple reflections of shocks and rarefactions off
the walls and many interactions of waves inside the domain. We
use here the same node setup as in the previous tests.

We compute a reference solution, this
time using a resolution of $n=16000$ points and compare the
performance of the results setting $n=800$ with the
WENO5-RK3, WENO5-LW5 and WENO5-LWA5 schemes with CFL=0.5. The
results of the density field are shown in Figure \ref{bw}, where the
conclusions are the same than those obtained in the
previous experiments.
\begin{figure}[htb]
  \centering
  \begin{tabular}{cc}
    \includegraphics[width=0.45\textwidth]{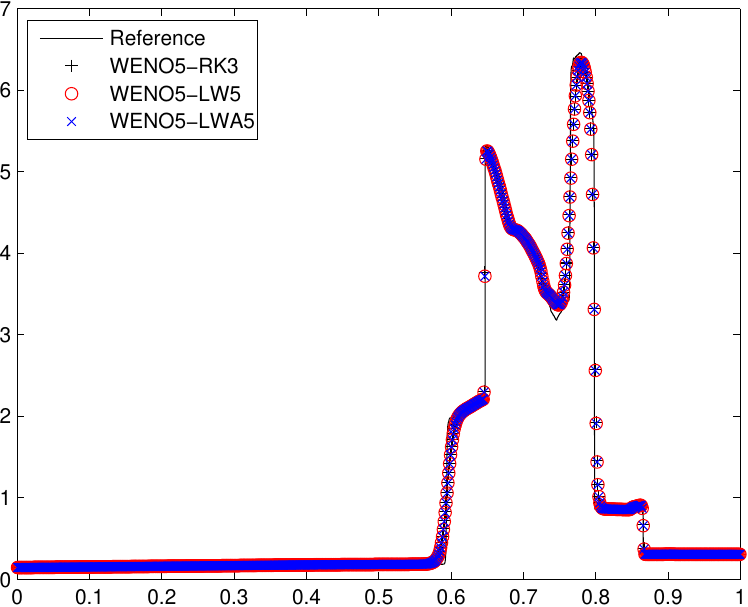}
    & \includegraphics[width=0.45\textwidth]{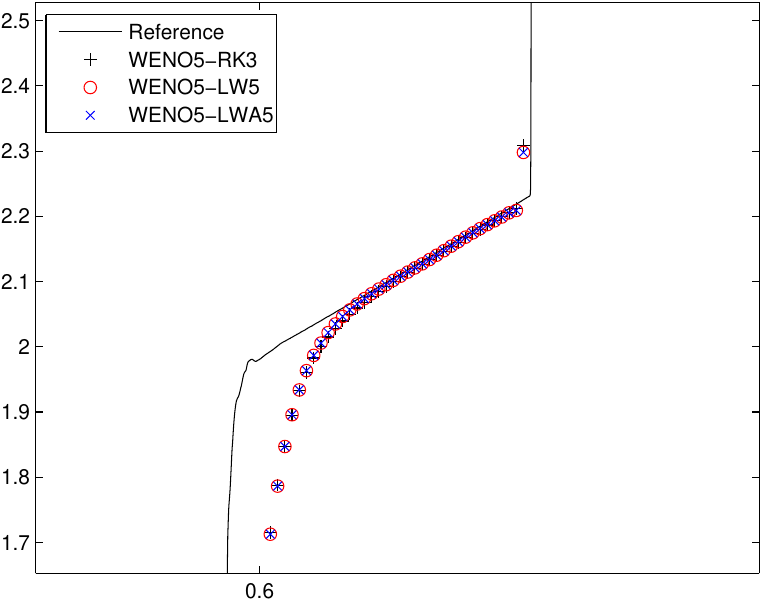} \\
    (a) Density field & (b) Density field (zoom) \\
    \includegraphics[width=0.45\textwidth]{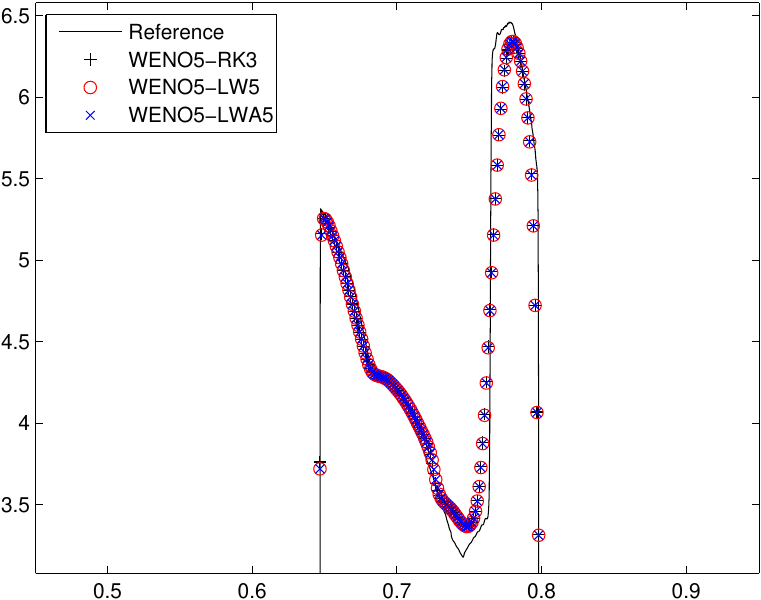}
    & \includegraphics[width=0.45\textwidth]{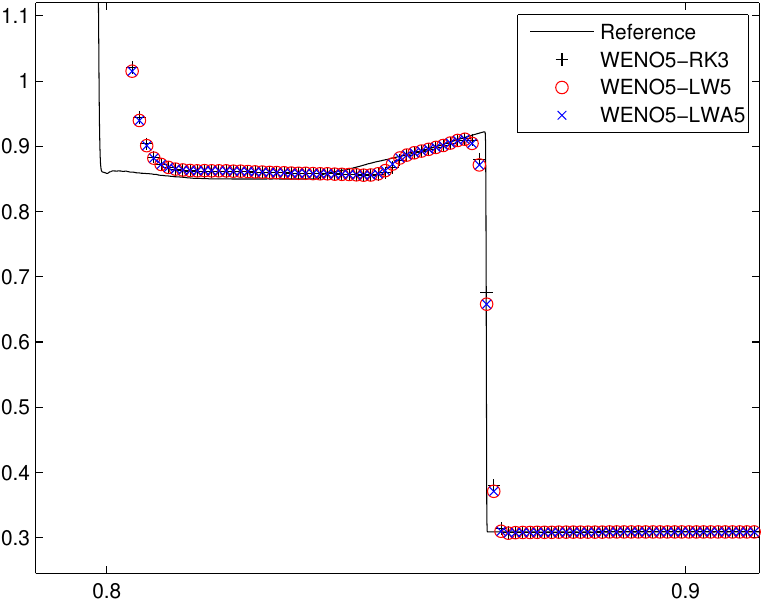} \\
    (c) Density field (zoom) & (d) Density field (zoom) \\
  \end{tabular}
  \caption{Blast wave.}
  \label{bw}
\end{figure}

\subsection{2D Euler equations}
The equations that will be considered in this section are the
two-dimensional Euler equations for inviscid gas dynamics
\begin{equation}\label{eq:eulereq}
\begin{aligned}
&u_t+f(u)_x+g(u)_y=0,\quad u=u(x,y,t),\\
&u=\left[\begin{array}{c}
    \rho \\
    \rho v^x \\
    \rho v^y \\
    E \\
  \end{array}\right],\hspace{0.3cm}f(u)=\left[\begin{array}{c}
    \rho v^x \\
    p+\rho (v^x)^2 \\
    \rho v^xv^y \\
    v^x(E+p) \\
  \end{array}\right],\hspace{0.3cm}g(u)=\left[\begin{array}{c}
    \rho v^y \\
    \rho v^xv^y \\
    p+\rho (v^y)^2 \\
    v^y(E+p) \\
  \end{array}\right].
\end{aligned}
\end{equation}
In these equations,  $\rho$ is the density, $(v^x, v^y)$  is the
velocity and  $E$ is the specific energy of the system. The variable
$p$  stands for the pressure and is given by the equation of state:
$$p=(\gamma-1)\left(E-\frac{1}{2}\rho((v^x)^2+(v^y)^2)\right),$$
where $\gamma$ is the adiabatic constant, that will be taken as
$1.4$ in all the experiments as in the 1D case.
\subsubsection{Smooth solution}
In order to test the accuracy of our scheme
in the general scenario of a multidimensional system of conservation
laws, we perform a test using the 2D Euler equations with smooth
initial conditions, given by
\begin{equation*}
  \begin{split}
    u_0(x,y)&=(\rho(x,y),v^x(x,y),v^y(x,y),E(x,y))\\
    &=\left(\frac{3}{4}+\frac{1}{2}\cos(\pi(x+y)),
    \frac{1}{4}+\frac{1}{2}\cos(\pi(x+y)),\right.\\
    &\left.\frac{1}{4}+\frac{1}{2}\sin(\pi(x+y)),
    \frac{3}{4}+\frac{1}{2}\sin(\pi(x+y))\right),
  \end{split}
\end{equation*}
where $x\in\Omega=[-1,1]\times[-1,1]$, with periodic boundary
conditions.

In order to perform the smoothness analysis, we compute a reference
solution in a fine mesh and then compute numerical solutions for the
resolutions $n\times n$, for $n=10\cdot 2^k,$
$1\leq k\leq 5$, obtaining the results shown in Tables
\ref{SE2-CK}-\ref{SE2-CKR} at the time $t=0.025$ for CFL=0.5.
\begin{table}[htb]
  \centering
  \begin{tabular}{|c|c|c|c|c|}
    \hline
    $n$ & Error $\|\cdot\|_1$ & Order $\|\cdot\|_1$ & Error
    $\|\cdot\|_{\infty}$ & Order $\|\cdot\|_{\infty}$ \\
    \hline
    40 & 1.80E$-5$ & $-$ & 2.74E$-4$ & $-$ \\
    \hline
    80 & 1.09E$-6$ & 4.05 & 1.80E$-5$ & 3.93 \\
    \hline
    160 & 3.89E$-8$ & 4.80 & 7.36E$-7$ & 4.61 \\
    \hline
    320 & 1.29E$-9$ & 4.92 & 2.49E$-8$ & 4.88 \\
    \hline
    640 & 4.11E$-11$ & 4.97 & 8.07E$-10$ & 4.95 \\
    \hline
    1280 & 1.23E$-12$ & 5.06 & 2.43E$-11$ & 5.06  \\
    \hline
  \end{tabular}
  \caption{Error table for 2D Euler equation, $t=0.025$. WENO5-LW5.}
  \label{SE2-CK}
\end{table}
\begin{table}[htb]
  \centering
  \begin{tabular}{|c|c|c|c|c|}
    \hline
    $n$ & Error $\|\cdot\|_1$ & Order $\|\cdot\|_1$ & Error
    $\|\cdot\|_{\infty}$ & Order $\|\cdot\|_{\infty}$ \\
    \hline
    40 & 1.80E$-5$ & $-$ & 2.74E$-4$ & $-$ \\
    \hline
    80 & 1.09E$-6$ & 4.05 & 1.80E$-5$ & 3.93 \\
    \hline
    160 & 3.89E$-8$ & 4.80 & 7.36E$-7$ & 4.61 \\
    \hline
    320 & 1.29E$-9$ & 4.92 & 2.49E$-8$ & 4.88 \\
    \hline
    640 & 4.11E$-11$ & 4.97 & 8.07E$-10$ & 4.95 \\
    \hline
    1280 & 1.23E$-12$ & 5.06 & 2.43E$-11$ & 5.06  \\
    \hline
  \end{tabular}
  \caption{Error table for 2D Euler equation, $t=0.025$. WENO5-LWA5.}
  \label{SE2-CKR}
\end{table}

We can see thus that our scheme achieves the desired accuracy even in
the general scenario of a multidimensional system of conservation
laws, which is consistent with our theoretical results. Also, we can
see that again the results obtained through the approximate
Lax-Wendroff-type procedure are almost the same than those obtained using
the exact version.

\subsubsection{Double Mach Reflection}
This experiment uses the Euler equations to model a vertical right-going Mach
10 shock colliding with an equilateral triangle. By symmetry, this is
equivalent to a collision with a ramp with a slope of 30 degrees with
respect to the horizontal line.

For the sake of simplicity, we consider the equivalent problem in a
rectangle, consisting on a rotated shock, whose vertical angle is
$\frac{\pi}{6}$ rad. The domain we consider in this
problem is the rectangle $\Omega=[0,4]\times[0,1]$, whose initial conditions
are
$$u_0(x,y)=\begin{cases}
  C_1 & y\leq\frac{1}{4}+\tan(\frac{\pi}{6})x,\\
  C_2 & y>\frac{1}{4}+\tan(\frac{\pi}{6})x,\\
\end{cases}$$
where
\begin{equation*}
  \begin{split}
    C_1=(\rho_1,v^x_1,v^y_1,E_1)^T
    &=(8,8.25\cos(\frac{\pi}{6}),-8.25\sin(\frac{\pi}{6}),563.5)^T,\\
    C_2=(\rho_2,v^x_2,v^y_2,E_2)^T
    &=(1.4,0,0,2.5)^T.\\
  \end{split}
\end{equation*}
We impose inflow boundary conditions, with value $C_1$, at the left
side, $\{0\}\times[0,1]$, outflow boundary conditions both at
$[0,\frac{1}{4}]\times\{0\}$ and $\{4\}\times[0,1]$, reflecting
boundary conditions at  $]\frac{1}{4},4]\times\{0\}$ and inflow
boundary conditions at the upper side, $[0,4]\times\{1\}$, which
mimics the shock at its actual traveling speed:
$$u(x,1,t)=\begin{cases}
  C_1 & x\leq\frac{1}{4}+\frac{1+20t}{\sqrt{3}}, \\
  C_2 & x>\frac{1}{4}+\frac{1+20t}{\sqrt{3}}.\\
\end{cases}$$
We run different simulations until $t=0.2$ at a resolution of
$2048\times512$ points
for CFL=0.4 and a different combination of techniques, involving
WENO5-RK3, WENO5-LW5 and WENO5-LWA5.

The results are presented as a Schlieren plot of the turbulence zone
and they are shown in Figure \ref{dmr}.

\begin{figure}[htb]
\centering
\begin{tabular}{cc}
\multicolumn{2}{c}{\includegraphics[width=0.45\textwidth]{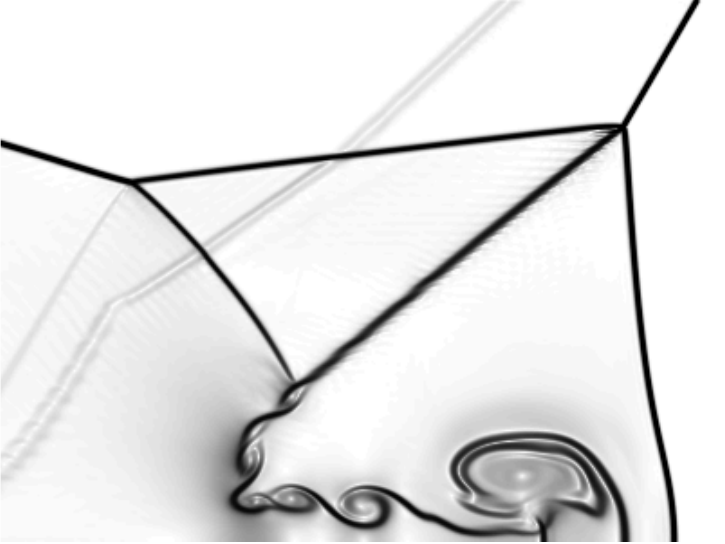}}\\
\multicolumn{2}{c}{(a) WENO5-RK3.}\\[1ex]
\includegraphics[width=0.45\textwidth]{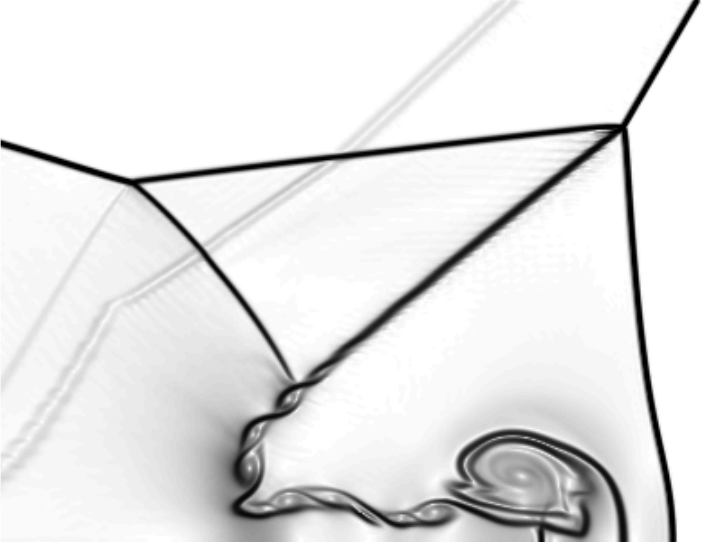} &
\includegraphics[width=0.45\textwidth]{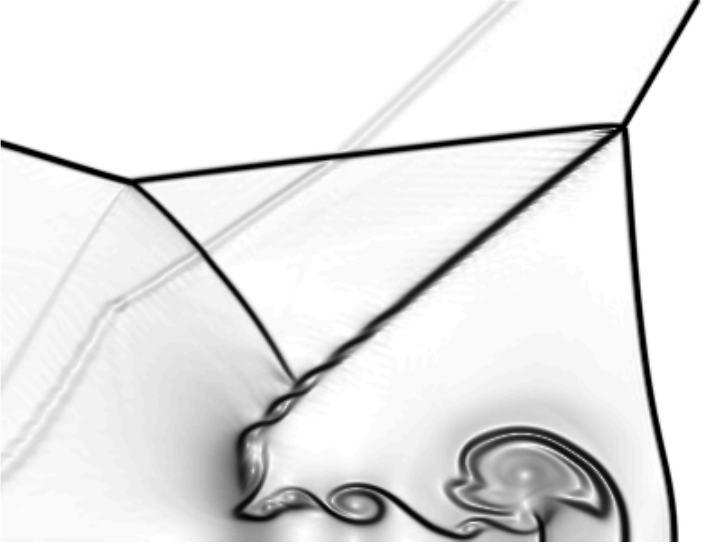} \\
(b) WENO5-LW5. & (c) WENO5-LWA5.
\end{tabular}
\caption{Double Mach Reflection. Pressure field.}
\label{dmr}
\end{figure}

From Figure \ref{dmr} it can be concluded that the results obtained
through the exact and approximate Lax-Wendroff-type techniques are again
quite similar.

Finally, in order to illustrate that the LW techniques are more
efficient than the RK3 time discretization, we show a performance test
involving the computational time required by each technique by running
the Double Mach Reflection problem for the resolution
$200\times50$. The results are shown in Table \ref{performance}, where
the field ``Efficiency'' stands for
$\displaystyle\frac{t_{\text{RK3}}}{t_{\text{LW*}}}$.
\begin{table}
  \centering
  \begin{tabular}{|c|c|}
    \hline
    Scheme & Efficiency \\
    \hline
    WENO5-LW5 & $1.44$ \\
    \hline
    WENO5-LWA5 & $1.54$ \\
    \hline
  \end{tabular}
  \caption{Performance table.}
  \label{performance}
\end{table}

We see that in this case both fifth order accurate in time techniques
are even faster than the third order accurate in time Runge-Kutta
method. Moreover, we see that in this case our proposed scheme with
approximate fluxes is faster than the exact formulation. On the
  other hand, it must be remarked that the CFL limits that we have
  empirically obtained for the LW and LWA 
  technique are  identical, therefore the
table entails a fair comparison between both techniques.

\section{Conclusions}\label{cnc}

In this paper we have presented an alternative to the Lax-Wendroff-type
procedure based on computing approximations of the time derivatives of
the fluxes rather than obtaining the exact expressions, which requires
symbolic manipulations in most of cases and yields a severe
computational cost as the order of the scheme increases, since its
growth rate is exponential.

We have proven both theoretically and
numerically that our scheme achieves the desired accuracy and that it
is conservative. Numerical experiments show that the flux
approximation procedure yields essentially the same results as the
exact formulation.

In terms of a comparison with the Runge-Kutta methods, the proposed
scheme keeps the same benefits of the original one as well, since only
one local characteristic decomposition is required per cell interface
and time step, yielding a better performance than the Runge-Kutta
methods, even when higher order accuracy in time is used, as it has been seen
in the numerical experiments section.

Therefore, as a concluding remark, we have obtained a scheme which
does require no computation of flux derivatives, yielding an easy
implementation, similar to the required for Runge-Kutta methods, as
well as a better performance in terms of computational cost in cases
where the exact time derivatives of the flux have large
expressions, which in turn keeps the same benefits of the original
Lax-Wendroff-type procedure. All the aforementioned benefits are combined,
as stated before, with the fact that quality of the numerical results
is essentially the same than those obtained with the exact procedure.

\appendix

\section{Appendix}
\label{appendixa}

For the sake of completeness, we prove
Theorem \ref{th:faadibruno} in this appendix, for we have not found
satisfactory references for its proof.

The following result is easily established.
\begin{lemma}\label{lemma:3}
Assume $T\colon \mathbb R^n \to \mathcal{M}(s, n)$ is differentiable
(equivalently, $T_{i_{1},\dots,i_{i_s}}$ are differentiable) and that
$A\colon \mathbb R\to \mathbb R^{n\times s}$,
$u:\mathbb{R}\rightarrow\mathbb{R}^n$ are also differentiable.
Then, $\forall x\in\mathbb{R}$
\begin{equation*}
  \frac{d}{dx}T(u(x))A(x)=T'(u(x))[u'(x)\ A(x)]+T(u(x))\sum_{j=1}^{s} d_j A(x),
\end{equation*}  
where we have used the notation
$d_j A(x)$ for the $n\times s$ matrix given by the columns:
\begin{equation*}
  (d_j A(x))_{k}=\begin{cases}
    A_{k}(x) & k\neq j\\
    A'_{j}(x) & k= j
  \end{cases}    
\end{equation*}
\end{lemma}

  We introduce some further notation for the proof of Theorem
  \ref{th:faadibruno}. 
For $s\in\mathbb N$, we denote
\begin{align*}
  \mathcal{P}_{s,j}&=\{ m\in\mathcal{P}_s /
m_j\neq 0 \}.
\end{align*}  
We denote also
\begin{equation*}
  S_0\colon \mathcal{P}_{s}\to \mathcal{P}_{s+1,1},\quad
  S_0(m)_k=
  \begin{cases}
    0 & k=s+1\\
    m_k& s\geq k\neq 1\\
    m_{1}+1 & k=1,
  \end{cases}    
\end{equation*}  
\begin{equation*}
  S_j\colon \mathcal{P}_{s,j}\to \mathcal{P}_{s+1,j+1},\quad
  S_j(m)_k=
  \begin{cases}
    0 & k=s+1\\
    m_k& s\geq k\neq j, j+1\\
    m_{j}-1 &s\geq  k=j\\
    m_{j+1}+1 &s\geq  k=j+1.
  \end{cases}    
\end{equation*}
for $1\leq j<s$,  and $S_s$ that maps $(0,\dots,0,1)\in\mathbb N^s$ to $(0,\dots,0,1)\in\mathbb N^{s+1}$.

\begin{proof}{(of Theorem \ref{th:faadibruno})}
We use induction on $s$, the case $s=1$ being the chain rule. By the
induction hypothesis for $s$ and Lemma
\ref{lemma:3} we deduce:
\begin{gather*}
  \frac{d^{s+1}f(u(x))}{dx^{s+1}} =
  \sum_{m\in\mathcal{P}_s} \mn \frac{d}{dx}\left(f^{(|m|)}(u(x)) D^m u(x)\right)\\
  =\sum_{m\in\mathcal{P}_s} \mn \big((f^{(|m|)})'(u(x)) [u'(x)\ D^m u(x)]+
  f^{(|m|)}(u(x)) \sum_{j=1}^{n}d_j D^m u(x)\big)\\
  =\sum_{m\in\mathcal{P}_s} \mn \big(f^{(|m|+1)}(u(x)) [u'(x)\ D^m u(x)]+
  f^{(|m|)}(u(x)) \sum_{j=1}^{n}d_jD^m u(x)\big).
\end{gather*}
Now,
\begin{equation*}
  d_j D^{m} u(x)=D^{S_j(m)} u(x) P E,
\end{equation*}  
where $P$ is a permutation matrix corresponding to the transposition of
$j$ and $\sum_{l\leq k} m_l$, with $\sum_{l< k} m_l < j \leq
\sum_{l\leq k} m_l$ and $E$ is a diagonal matrix with $k+1$
in the $\sum_{l\leq k} m_l$ entry and 1 in the rest.

 By the symmetry of $f^{(|m|)}$, if  $\sum_{l< k} m_l < j \leq
\sum_{l\leq k} m_l$
\begin{equation*}
  f^{(|m|)}(u(x)) d_jD^m u(x)=
  (k+1)
  f^{(|S_{k}(m)|}(u(x)) D^{S_{k}(m)} u(x),
\end{equation*}
therefore, collecting identical terms,
\begin{multline*}
    \frac{d^{s+1}f(u(x))}{dx^{s+1}} =
    \sum_{m\in\mathcal{P}_s} \mn \big(f^{(|S_0 (m)|)}(u(x)) D^{S_0
                                      (m)}u(x)\\
  +
    \sum_{j=1}^{n}f^{(|m|)}(u(x)) d_jD^{S_j (m)} u(x)\big)
  \end{multline*}
  can be written as
\begin{equation}\label{eq:4f}
  \begin{aligned}
        \frac{d^{s+1}f(u(x))}{dx^{s+1}} &=
    \sum_{m\in\mathcal{P}_s} \mn \big(f^{(|S_0 (m)|)}(u(x)) D^{S_0
      (m)}u(x)\\
    &+
    \sum_{k=1}^{n}  m_k  (k+1)   f^{(|S_{k}(m)|)}(u(x)) D^{S_{k}(m)} u(x)\big),
  \end{aligned}
\end{equation}
where we point out that in the last expression
  the only terms that actually appear are those for which
  $m_k>0$. Since $m_k-1=(S_{k}(m))_{k}$, by collecting the terms for $m,
  k$ such that $S_k(m)=\widehat m$, \eqref{eq:4f} can be written as
\begin{align}\label{eq:5f}
  \frac{d^{s+1}f(u(x))}{dx^{s+1}} =
  \sum_{\widehat m\in\mathcal{P}_{s+1}} 
  a_{\widehat m} f^{(|\widehat m|)}(u(x)) D^{\widehat m} u(x),
  \end{align}
  where
  \begin{equation}\label{eq:452}
      a_{\widehat m}=
\begin{cases}
\widetilde{a_{\widehat m}} & \mbox{if } \widehat{m_1}=0 \\
\widetilde{a_{\widehat m}}+\left[
    \begin{array}{c}
      s\\
      S_0^{-1}(\widehat{m})
    \end{array}      
    \right] & \mbox{if } \widehat{m_1}\neq0,
\end{cases}
              \quad
    \widetilde{a_{\widehat m}}=\sum_{\scriptsize\begin{array}{c}\widehat m=S_k (m), \\ k\in\{1,\dots,s\},\\
    m\in\mathcal{P}_{s,k}
  \end{array}
} \mn  m_k  (k+1). 
\end{equation}
For $k\in\{1,\dots,s\}$, and   $ m\in\mathcal{P}_{s,k}$, such that $\widehat m=S_k
(m)$, i.e., $\widehat m_i=m_i$, $i\neq k, k+1$, $\widehat m_{k}=m_{k}-1$,
$\widehat m_{k+1}=m_{k+1}+1$,  we deduce:
\begin{align*}
\mn  m_k  (k+1)&= \frac{s!}{m_1!\dots (m_{k}-1)! m_{k+1}!\dots
  m_s!}(k+1)\\
&=
\frac{s!}{\widehat m_1!\dots \widehat m_{k}! (\widehat m_{k+1}-1)!\dots
  \widehat m_s!}(k+1)\\
&=
\frac{s!}{\widehat m_1!\dots \widehat m_{k}! \widehat m_{k+1}!\dots \widehat m_s!}\widehat
m_{k+1}(k+1).
\end{align*}
Let $\widehat m=S_k (m)$ with  $k<s$, then  one has
$\widehat m_{s+1}=0$.  The only element $m\in\mathcal{P}_{s,s}$ is
$(0,\dots,0,1)\in\mathbb N^s$ and $S_s(m)=(0,\dots,0,1)\in\mathbb N^{s+1}$.
Therefore 
  \begin{align}
    \notag
    \widetilde{a_{\widehat m}}&=\frac{s!}{\widehat m_1!\dots \widehat  m_{s+1}!}\sum_{\scriptsize\begin{array}{c}\widehat m=S_k (m), \\ k\in\{1,\dots,s\},\\
    m\in\mathcal{P}_{s,k}
  \end{array}
} \widehat m_{k+1}(k+1)\\
    \label{eq:450}
    \widetilde{a_{\widehat m}}
&=\frac{s!}{\widehat m_1!\dots \widehat m_{s+1}!}\sum_{k=1}^{s} \widehat
m_{k+1}(k+1)
=\frac{s!}{\widehat m_1!\dots \widehat m_{s+1}!}\sum_{k=2}^{s+1} \widehat m_{k}k.
\end{align}

 On the other hand, if $\widehat{m}_1\neq0$, then:
 \begin{equation}\label{eq:451}
   \left[
    \begin{array}{c}
      s\\
      S_0^{-1}(\widehat{m})
    \end{array}      
    \right]
    =\frac{s!}{(\widehat{m}_1-1)!\widehat{m}_2!\cdots\widehat{m}_s!}
    =\frac{s!}{\widehat{m}_1!\widehat{m}_2!\cdots\widehat{m}_s!
      \widehat{m}_{s+1}!}\widehat{m}_1,
  \end{equation}
  where the last equality holds since, as before, we have
 $\widehat{m}_{s+1}=0$. Then, regardless of $\widehat{m}_1$, \eqref{eq:450} and
 \eqref{eq:451} yield for $\widehat m\in\mathcal{P}_{s+1}$
  \begin{equation}\label{eq:453}
    a_{\widehat m}=\frac{s!}{\widehat m_1!\dots \widehat m_{s+1}!}\sum_{k=1}^{s+1} \widehat
    m_{k}k=\frac{s!}{\widehat m_1!\dots\widehat m_{s+1}!}(s+1)=\mna{s+1}{\widehat m},
  \end{equation}
  since $\widehat m\in\mathcal{P}_{s+1}$ means $\sum_{k=1}^{s+1} \widehat
m_{k}k=s+1$. We deduce from \eqref{eq:5f}, \eqref{eq:452} and \eqref{eq:453}
that
\begin{align*}
  \frac{d^{s+1}f(u(x))}{dx^{s+1}} =\sum_{\widehat m\in \mathcal{P}_{s+1}}
\mna{s+1}{\widehat m}
    f^{(|\widehat m|)}(u(x)) D^{\widehat m}u(x),
  \end{align*}
  which concludes the proof by induction.
\end{proof}

\section{Appendix}
\label{appendixb}
We include here the proof of Proposition \ref{prop:1}.
\begin{proof}
  For the accuracy analysis of the local truncation error, we take
  \begin{equation}\label{eq:45}
\widetilde u^{(0)}_{i,n}=u(x_i, t_n).
\end{equation}
  We now use induction on $k=1,\dots,R$ to prove that
  \begin{align}
    \label{eq:47}
\widetilde u^{(k)}_{i,n}&= u^{(k)}_{i,n}+c^{k}(x_i, t_n)h^{R-k+1}+\bigO(h^{R-k+2}),
\end{align}
for continuously differentiable functions $c^k$. The result  in
\eqref{eq:47} for $k=1$ immediately follows from the fact that
WENO finite differences  applied to the exact data in \eqref{eq:45}
yield approximations 
  \begin{equation*}
    \frac{\hat f_{i+\mig,n}-\hat f_{i-\mig,n}}{h}=f(u)_{x}(x_{i},
    t_n)+\widetilde c^{1}(x_i, t_n)h^{2r-1}+\bigO(h^{2r}),\quad r=\nu.
  \end{equation*}
  From the definition in \eqref{eq:151} we deduce
  \begin{equation*}
    \widetilde u^{(1)}_{i,n}=-\frac{\hat f_{i+\mig,n}-\hat
      f_{i-\mig,n}}{h}=u^{(1)}_{i,n}-\widetilde c^{1}(x_i,
    t_n)h^{2r-1}+\bigO(h^{2r}), \quad 2r\geq R+1,
  \end{equation*}
  thus proving the case $k=1$, by taking $c^1=-\widetilde c^1$ if
  $2r=R+1$ or $c^1=0$ if $2r>R+1$.

Assume now the result to hold for $k$ and aim to prove it for
$k+1\leq R$. For this purpose we first prove the following estimate:
  \begin{align}
    \label{eq:46}
\widetilde f^{(k)}_{i,n}&= f^{(k)}_{i,n}+a^k(x_i, t_n)h^{R-k}+b^k(x_i, t_n)h^{R-k+1}+\bigO(h^{R-k+2}),
\end{align}
for  continuously differentiable functions $a^k, b^k$.

From
\eqref{eq:15} and \eqref{eq:3}, with
$q=\left\lceil   \frac{R-k}{2}\right\rceil$
and the notation $v=T_k[h, i,     n]$, for fixed $i,n$: 
  \begin{align}\label{eq:1}
    \widetilde f^{(k)}_{i,n} &=
    (f(v))^{(k)}(0)+\alpha^{k,q}    (f(v))^{(k+2q)}(0)h^{2q} + \bigO(h^{2q+2}).
  \end{align}
  Now, Fa\`a di Bruno's formula \eqref{eq:faadibruno} and the fact
  that  \eqref{eq:taylor} implies $v^{(p)}(0)=\widetilde u^{(p)}_{i,n}$, $p=1,\dots,k$, yields:
\begin{equation}     \label{eq:5}
   \begin{aligned}
    (f(v))^{(k)}(0)&=\sum_{s\in \mathcal{P}_{k}}
    \left(
      \begin{array}{c}
        k\\
        s
      \end{array}
    \right)
    f^{(|s|)}(v(0))       \big(D^{s} v(0)\big),\\
    D^{s} v (0)&=
    \begin{bmatrix}
      \overbrace{
        \begin{array}{ccc}
          \frac{v^{(1)}(0)}{1!}
          &\dots&
          \frac{v^{(1)}(0)}{1!}
        \end{array}
      }^{s_1}
      &\dots&
      \overbrace{
        \begin{array}{ccc}
          \frac{v^{(k)}(0)}{k!}
          &\dots&
          \frac{v^{(k)}(0)}{k!}
        \end{array}
      }^{s_k}
    \end{bmatrix}=      \\
   &= 
    \begin{bmatrix}
      \overbrace{
        \begin{array}{ccc}
          \frac{\widetilde u_{i,n}^{(1)}}{1!}
          &\dots&
          \frac{\widetilde u_{i,n}^{(1)}}{1!}
        \end{array}
      }^{s_1}
      &\dots&
      \overbrace{
        \begin{array}{ccc}
          \frac{\widetilde u_{i,n}^{(k)}}{k!}
          &\dots&
          \frac{\widetilde u_{i,n}^{(k)}}{k!}
        \end{array}
      }^{s_k},
    \end{bmatrix}      
  \end{aligned}       
\end{equation}
where we remark that $D^{s} v (0)$ is an $m\times |s|$ matrix formed
by the columns $\frac{\widetilde u_{i,n}^{(p)}}{p!}$ appearing in the previous
expression. 
Since $v=T_k[h, i, n]$ is a $k$-th degree
polynomial, $v^{(j)}=0$  for $j>k$. Therefore, in the same
fashion as before,  
\begin{align}
    \label{eq:57}
    &(f(v))^{(k+2q)}(0)=\sum_{s\in \mathcal{P}^{k}_{k+2q}}
    \left(
      \begin{array}{c}
        k\\
        s
      \end{array}
    \right)
    f^{(|s|)}(v(0))       \big(D^{s} v(0)\big)
  \end{align}
  where $\mathcal{P}^{k}_{k+2q}=\{s\in \mathcal{P}_{k+2q} / s_j=0, j>k\}$.
  
  On the other      hand, another application of Fa\`a di Bruno's
  formula to $f(u)$, yields:
  \begin{equation}\label{eq:4}
  \begin{aligned}
    f_{i,n}^{(k)}&=f(u)^{(k)}(x_i, t_n)=
    \sum_{s\in \mathcal{P}_{k}}
    \left(
      \begin{array}{c}
        k\\
        s
      \end{array}
    \right)
    f^{(|s|)}(u(x_i, t_n))       \big(D^{s} u(x_i, t_n)\big)\\
    D^{s} u (x_i, t_n)&=
    \begin{bmatrix}
      \overbrace{
        \begin{array}{ccc}
          \frac{u^{(1)}(x_i, t_n)}{1!}
          &\dots&
          \frac{u^{(1)}(x_i, t_n)}{1!}
        \end{array}
      }^{s_1}
      &\dots&
      \overbrace{
        \begin{array}{ccc}
          \frac{u^{(k)}(x_i, t_n)}{k!}
          &\dots&
          \frac{u^{(k)}(x_i, t_n)}{k!}
        \end{array}
      }^{s_k}
    \end{bmatrix}      \\
    &=
    \begin{bmatrix}
      \overbrace{
        \begin{array}{ccc}
          \frac{u_{i,n}^{(1)}}{1!}
          &\dots&
          \frac{u_{i,n}^{(1)}}{1!}
        \end{array}
      }^{s_1}
      &\dots&
      \overbrace{
        \begin{array}{ccc}
          \frac{u_{i,n}^{(k)}}{k!}
          &\dots&
          \frac{u_{i,n}^{(k)}}{k!}
        \end{array}
      }^{s_k}
    \end{bmatrix}      
  \end{aligned}
\end{equation}
We have $ v(0)=\widetilde
  u_{i,n}^{(0)}=u(x_i, t_n)$ and, by induction,
  \begin{equation}\label{eq:90}
    \widetilde
  u_{i,n}^{(l)}=u_{i,n}^{(l)}+c^{l}(x_{i},
  t_{n})h^{R-l+1}+\bigO(h^{R-l+2}),\quad l=1,\dots,k.
\end{equation}
For any $s\in\mathcal{P}_{k}$, $D^{s}v(0)$ is a
  $m\times |s|$ matrix, and  for any
  $\mu\in\{1,\dots,m\}$ and 
  $\nu\in\{1,\dots,|s|\}$, we have from \eqref{eq:ds}, \eqref{eq:5},
  \eqref{eq:4} and \eqref{eq:90} that 
  \begin{equation}\label{eq:78}
(D^s  v(0)-D^s u(x_i, t_n))_{\mu, \nu} = \frac{(\widetilde
    u_{i,n}^{(l)}-u_{i,n}^{(l)})_{\mu}}{l!} = \frac{c^l_{\mu}(x_i,
    t_n)}{l!} h^{R-l+1} + 
\bigO(h^{R-l+2}),
\end{equation}
for some $l=l(s, \nu)\leq k$. From the definition of the set
$\mathcal{P}_k$, the only $k$-tuple $s\in \mathcal{P}_{k}$ such that
$s_k\neq 0$ is $s^*=(0,\dots,1)$. Therefore, from the definition
of the operator
$D^s$ in \eqref{eq:ds} (or \eqref{eq:4} \eqref{eq:5}), the only $s\in\mathcal{P}_{k}$, $\nu\leq |s|$,
such that $l(s, \nu)=k$ is $s^*$, $\nu=|s^*|=1$. We deduce from \eqref{eq:78} that
\begin{align}\label{eq:79}
(D^s  v(0)-D^s u(x_i, t_n))_{\mu, \nu}=\bigO(h^{R-k+2}), \quad \forall 
s\in\mathcal{P}_s, s\neq s^*, \forall \mu\leq m, \forall 
  \nu\leq |s|\\
  \label{eq:80}
(D^{s^*}  v(0)-D^{s^*} u(x_i, t_n))_{\mu, 1}=\frac{c^k_{\mu}(x_i,
    t_n)}{k!} h^{R-k+1} + 
\bigO(h^{R-k+2}),
\end{align}  
We deduce from \eqref{eq:78}, \eqref{eq:77},
\eqref{eq:5}, \eqref{eq:4}, \eqref{eq:79} that
  \begin{align}\notag
    f(v)^{(k)}(0)-f(u)^{(k)}(x_i, t_n)&=
    \left(
      \begin{array}{c}
        k\\
        s^*
      \end{array}
    \right)
    f^{(|s^*|)}(u(x_i, t_n))       \big(D^{s^*} v(0)-D^{s^*} u(x_i,
    t_n)\big)  \\
    \notag
    &+
    \sum_{s\in \mathcal{P}_{k}, s\neq s^*}
    \left(
      \begin{array}{c}
        k\\
        s
      \end{array}
    \right)
    f^{(|s|)}(u(x_i, t_n))       \big(D^{s} v(0)-D^{s} u(x_i,
    t_n)\big)\\
    \label{eq:82}
    &=\sum_{\mu=1}^{m}\frac{\partial f}{\partial u_{\mu}}(u(x_i,
    t_n))c^k_{\mu}(x_i,    t_n) h^{R-k+1} + \bigO(h^{R-k+2}), 
  \end{align}
  where we have collected the order $R-k+1$ leading terms originating
  from the first term associated to the $k$-tuple $s^*=(0,\dots,0,1)$.
  With a similar argument, taking into account that $k+1\leq R$, we
  deduce from \eqref{eq:57} and \eqref{eq:78} that 
  \begin{equation}\label{eq:81}
    \begin{aligned}
    (f(v))^{(k+2q)}(0)&=e^{k,q}(x_i, t_n)+\bigO(h^{R-k+1})=e^{k,q}(x_i,
    t_n)+\bigO(h^2),\\
    e^{k,q}(x, t)&=\sum_{s\in \mathcal{P}^{k}_{k+2q}}
    \left(
      \begin{array}{c}
        k\\
        s
      \end{array}
    \right)
    f^{(|s|)}(v(0))       \big(D^{s} u(x, t)\big).
  \end{aligned}
\end{equation}
  
  Now, \eqref{eq:1}, \eqref{eq:4}, \eqref{eq:81} and \eqref{eq:82} yield:
  \begin{equation*}
 \widetilde f^{(k)}_{i,n}
 -f^{(k)}_{i,n}=\sum_{\mu=1}^{m}\frac{\partial f}{\partial
   u_{\mu}}(u(x_i, t_n))c^k_{\mu}(x_i,    t_n) h^{R-k+1} +
 \bigO(h^{R-k+2}) + e^{k,q}(x_i, t_n)h^{2q}+\bigO(h^{2q+2}).
\end{equation*}
Since $2q= R-k$ or $2q= R-k+1$, we deduce \eqref{eq:46} with
\begin{align*}
  a^{k}(x, t)&=\begin{cases}
    e^{k}(x, t) & 2q=R-k\\
    0& 2q=R-k+1\\
  \end{cases}\\
  b^{k}(x, t)&=\begin{cases}
    \sum_{\mu=1}^{m}\frac{\partial f}{\partial
   u_{\mu}}(u(x, t))c^k_{\mu}(x,    t) & 2q=R-k\\
    \sum_{\mu=1}^{m}\frac{\partial f}{\partial
   u_{\mu}}(u(x, t))c^k_{\mu}(x,    t) + e^{k, q}(x, t)& 2q=R-k+1.
  \end{cases}
\end{align*}

To prove \eqref{eq:47} for $k+1$, we apply the linear operator
  $-\Delta_h^{1,q}$, for $q=\left\lceil\frac{R-k}{2}\right\rceil$,
  to both sides  of  the already established equality \eqref{eq:46},
  taking into account \eqref{eq:3} and that $2q\geq R-k$:
  \begin{align*}
    \widetilde
    u^{(k+1)}_{i,n}&=-\Delta_h^{1,q}
    \widetilde f^{(k)}_{i+\cdot,n} \\
    &=-\Delta_h^{1,q}
    f_{i+\cdot,n}^{(k)}
    -h^{R-k}\Delta_h^{1,q}G_{x_i,h}(a^k(\cdot,    t_n))
    -h^{R-k+1}\Delta_h^{1,q}G_{x_i,h}(b^k(\cdot,    t_n))
    +\bigO(h^{R-k+1}) 
    \\
    &=-\Delta_h^{1,q}
    G_{x_i,h}\big(f(u)^{(k)}(\cdot, t_n)\big)
    -h^{R-k}(\frac{\partial a^k}{\partial x}(x_i, t_n)+
    \bigO(h^{2q}))\\
    &-h^{R-k+1}(\frac{\partial b^k}{\partial x}(x_i,t_n)+
    \bigO(h^{2q}))
    +\bigO(h^{R-k+1})
    \\
    &=-[f(u)^{(k)}]_{x}(x_i,
    t_n)-\alpha^{1,q}\frac{\partial^{k+2q+1}f(u)}{\partial
      x^{2q+1}\partial t^{k}}(x_i, t_n) h^{2q}+
\bigO(h^{2q+2})\\
    &-h^{R-k}\frac{\partial a^k}{\partial x}(x_i, t_n)    +\bigO(h^{R-k+1})
    \\
    &=u^{(k+1)}(x_i,  t_n)+c^{k+1}(x_i, t_n)h^{R-k}+\bigO(h^{R-k+1}),
  \end{align*}
  where
  \begin{equation*}
  c^{k+1}(x, t)=-\frac{\partial a^k}{\partial x}(x, t)-
  \begin{cases}
    0 & 2q > R-k\\
    \alpha^{1,q}\frac{\partial^{k+2q+1}f(u)}{\partial
      x^{2q+1}\partial t^{k}}(x, t) & 2q=R-k.
  \end{cases}    
\end{equation*}

The local truncation error is given by
  \begin{align*}
    &u_{i,n+1}^{(0)}-    
    \sum_{l=0}^R\frac{ (\Delta t)^l}{l!}\widetilde u_{i,n}^{(l)},
  \end{align*}
  where $\widetilde u_{i,n}^{(l)}$ are computed from
  $\widetilde u_{i,n}^{(0)}=u(x_i, t_n)$. Taylor expansion of the first
  term and the estimates in \eqref{eq:47} yield that the local
  truncation error is: 
  \begin{align*}
    &\sum_{l=1}^{R}\frac{(\Delta t)^l }{l!}(u_{i,n}^{(l)}
    -\widetilde u_{i,n}^{(l)})+\bigO(h^{R+1})\\
    &=\sum_{l=1}^{R}\frac{(\Delta t)^l }{l!}\bigO(h^{R-l+1})+\bigO(h^{R+1}) =\bigO(h^{R+1}),
  \end{align*}
  since $\Delta t$ is proportional to $h$.\qed
\end{proof}

\section*{Acknowledgments}
 This research was partially
  supported by Spanish MINECO grants    MTM 2011-22741 and MTM 2014-54388-P.


\begin{thebibliography}{10}
\providecommand{\url}[1]{{#1}}
\providecommand{\urlprefix}{URL }
\expandafter\ifx\csname urlstyle\endcsname\relax
  \providecommand{\doi}[1]{DOI~\discretionary{}{}{}#1}\else
  \providecommand{\doi}{DOI~\discretionary{}{}{}\begingroup
  \urlstyle{rm}\Url}\fi

\bibitem{DonatMarquina96}
Donat, R., Marquina, A.: Capturing shock reflections: An improved flux formula.
\newblock J. Comput. Phys. \textbf{125}, 42--58 (1996)


\bibitem{faadibruno1857}
Faà~di Bruno, C.F.: Note sur un nouvelle formule de calcul différentiel.
\newblock Quart. J. Math. \textbf{1}, 359--360 (1857)




\bibitem{Hairer1993}
Hairer, E., N\o rsett, S. P., Wanner, G.: Solving Ordinary Differential Equations I. 
\newblock Springer, second edition (1993)

\bibitem{Harten1987}
Harten, A., Engquist, B., Osher, S., Chakravarthy, S.R.: Uniformly high order
  accurate essentially non-oscillatory schemes, {III}.
\newblock J. Comput. Phys. \textbf{71}(2), 231--303 (1987)
 
 
 \bibitem{Hickernell2008} 
 Hickernell, F. J., Yang, S.: Explicit hermite interpolation polynomials via the cycle index with applications.
\newblock Int. J. Numer. Anal. Comp., \textbf{5} (3), 457--465 (2008)
 
\bibitem{JiangShu96}
Jiang, G.S., Shu, C.W.: Efficient implementation of {Weighted} {ENO} schemes.
\newblock J. Comput. Phys. \textbf{126}, 202--228 (1996)

\bibitem{QiuShu2003}
  Qiu, J., Shu, C.W.: Finite difference WENO schemes with
  {L}ax-{W}endroff-type time discretizations.
\newblock SIAM J. Sci. Comput. \textbf{24}(6), 2185--2198 (2003)

\bibitem{ShuOsher1988}
Shu, C.W., Osher, S.: Efficient implementation of essentially non-oscillatory
  shock-capturing schemes.
\newblock J. Comput. Phys. \textbf{77}, 439--471 (1988)

\bibitem{ShuOsher1989}
Shu, C.W., Osher, S.: Efficient implementation of essentially non-oscillatory
  shock-capturing schemes, {II}.
\newblock J. Comput. Phys. \textbf{83}(1), 32--78 (1989)

\bibitem{TiburceAbadie1850}
Tiburce Abadie, J.C.F.: Sur la différentiation des fonctions de
fonctions.
\newblock Nouvelles annales de mathématiques, journal des candidats
aux écoles polytechnique et normale. \textbf{9}(1), 119--125 (1850)

\bibitem{TiburceAbadie1852}
Tiburce Abadie, J.C.F.: Sur la différentiation des fonctions de fonctions. Séries de Burmann, de Lagrange, de Wronski.
\newblock Nouvelles annales de mathématiques, journal des candidats
aux écoles polytechnique et normale. \textbf{11}(1), 376--383 (1852)

\bibitem{Toro2009}
E.~F. Toro.
\newblock {\em {R}iemann solvers and numerical methods for fluid dynamics}.
\newblock Springer-Verlag, third edition (2009)

\bibitem{You2014}  
\newblock You, X., Zhao, J., Yang, H., Fang, Y., Wu, X.: Order conditions for RKN methods solving general second-order oscillatory systems.
\newblock Numer. Algor. \textbf{66}, 147--176 (2014)

\end{thebibliography}
\end{document}